\documentclass[hidelinks,onefignum,onetabnum]{siamart250211}

\usepackage{hyphenat}

\hypersetup{%
	colorlinks=true,
	linkcolor=blue,
	linkbordercolor={0 0 1},
	citecolor=cyan,
}

\usepackage{lipsum}
\usepackage{amsfonts}
\usepackage{graphicx}
\usepackage{epstopdf}
\usepackage{algpseudocode,algorithm}
\usepackage{graphicx,subcaption,hyperref}
\ifpdf
  \DeclareGraphicsExtensions{.eps,.pdf,.png,.jpg}
\else
  \DeclareGraphicsExtensions{.eps}
\fi

\newsiamremark{remark}{Remark}
\newsiamremark{hypothesis}{Hypothesis}
\crefname{hypothesis}{Hypothesis}{Hypotheses}
\newsiamthm{claim}{Claim}
\newsiamremark{fact}{Fact}
\crefname{fact}{Fact}{Facts}

\headers{Steklov problem for Helmholtz}{N.Nigam and K. Patil and W. Sun}

\title{An Example Article\thanks{Submitted to the editors DATE.
\funding{This work was funded by the Fog Research Institute under contract no.~FRI-454.}}}

\author{Nilima Nigam\thanks{Department of Mathematics, Simon Fraser University, Burnaby, BC V5A 1S6, Canada.}
\and Kshitij Patil\footnotemark[1]~\thanks{Email: \email{kap15@sfu.ca}.}
\and Weiran Sun\footnotemark[1]}

\usepackage{amsopn}

\title{The spectrum of the Steklov-Helmholtz operator}

\renewcommand{\d}[0]{\text{d}}

\newcommand{\bb}[1]{\mathbb{#1}}

\renewcommand{\d}[0]{ \text{ d} }
\newcommand{\T}{ \mathcal{T} }

\newcommand{\dontshow}[1]{}

\begin{document}

\maketitle

\begin{abstract}
We present a wavenumber-robust strategy for computing Steklov eigenpairs of the Helmholtz operator $-\Delta -\mu^2$. As the wavenumber $\mu \rightarrow \mu_D$ from below  (where $\mu_D^2 $ is a Dirichlet-
Laplace eigenvalue of multiplicity $\ell$), the lowest $\ell$ Steklov-Helmholtz
eigenvalues diverge to $-\infty$. Computationally, the Steklov-Helmholtz eigenvalue problem becomes severely ill-conditioned when $\mu \approx \mu_D$.

We first reformulate the problem in terms of a suitably-defined Dirichlet-to-Neumann map. We then use an indirect approach based on a single layer ansatz.  The discrete single layer matrix is nearly singular
close to exceptional wavenumbers, and we use a reduced singular value decomposition to avoid the consequent ill-conditioning. For smooth domains, convergence of our eigenvalue solver is spectral. We use this method (called the BIO-MOD approach) for
shape optimization of scale-invariant Steklov-Helmholtz problems and prove that
the disk maximizes the second eigenvalue under appropriate scaling. For curvilinear polygons, we use polynomially-graded meshes rather than uniform meshes. As a proof of
concept, we also implemented BIO-MOD using RCIP quadratures (using the {\tt ChunkIE} implementation). The BIO-MOD approach successfully removes ill-conditioning near exceptional wavenumbers, and  very high eigenvalue accuracy (up to 10
digits for polygons, arbitrary precision accuracy for smooth domains) is observed.  

We deploy {our} approach to computationally study the spectral geometry of the Steklov-Helmholtz operator, including some questions about spectral asymptotics and spectral optimization.
\end{abstract}

\section{Introduction}
The spectra of elliptic operators are of considerable theoretical and practical interest. There is an extensive literature on accurately computing the Dirichlet and Neumann eigenvalues of the Laplacian, particularly for bounded domains.

The Steklov problem for the Laplace operator was first introduced by Steklov in a talk at the Kharkov  Mathematical society in December 1895, and was a major topic in his dissertation. In its simplest form, the spectral problem he considered can be stated as: 
{\it  On a bounded domain $\Omega\in \mathbb{R}^2$ find a non-trivial eigenfunction $u: \Omega \rightarrow \mathbb{R}$ and associated real eigenvalue $\sigma^L$  so that} 
\begin{equation}
	-\Delta u=0\quad \, \text{in}\, \Omega, \quad \text{and}\,\,  \frac{\partial u}{\partial n}  = \sigma^L \T(u) \, \quad \text{on the boundary}\, \, \Gamma.\label{eq:steklov-laplace}
\end{equation} Here $n$ is the unit outward normal to $\Gamma$, the boundary of the domain, and  $\T $ denotes the Dirichlet trace operator. 

We record some well-known facts concerning the {\it Steklov-Laplace} problem \autoref{eq:steklov-laplace}. First, it is clear the eigenfunctions are harmonic. Next, under relatively mild assumptions on the boundary smoothness, the Steklov-Laplace spectrum is countable and discrete, with $0=\sigma^L_0<\sigma^L_1\leq \sigma^L_2\leq...$, accumulating only at infinity. 

In this paper we present a computational strategy for computing the eigenvalues of the {\it Steklov-Helmholtz} operator for bounded Lipschitz domains $\Omega\subset \mathbb{R}^2$. The Steklov-Helmholtz problem for a fixed real {\it wavenumber $\mu$} can be stated as:\\
{\sc Problem I:} {\it  Find a suitably regular non-trivial function $u$ on $\Omega$ and an associated eigenvalue $\sigma \in \mathbb{R}$ so that }
\begin{equation}\label{eq:steklov-helmholtz}
		-\Delta u - \mu^2 u =0\, \quad \text{in} \, \,\Omega,\qquad 
		\frac{\partial u}{\partial n} =\sigma \T(u) \quad \text{on}\,\, \Gamma.
\end{equation} 

{\sc Problem I} presents some interesting features and challenges compared to \autoref{eq:steklov-laplace}.  First, solutions of {\sc Problem I} \eqref{eq:steklov-helmholtz} must satisfy the Helmholtz equation $-\Delta u - \mu^2u=0$ for a given fixed wavenumber $\mu$. We expect $u$ to oscillate on a scale set by $\mu$ as well as $\sigma$. If $\mu$ is large, these oscillations may need fine computational meshes even to resolve low eigenmodes. 

A further interesting feature of Problem I \eqref{eq:steklov-helmholtz} is that the eigenvalues $\sigma$ may be negative. This can immediately be seen by noting the Rayleigh quotient for \eqref{eq:steklov-helmholtz} may lose (semi-)definiteness:
\begin{align}\label{eq:RQ}
	\mathcal{R}[\mu,u]:= \frac{\int_{\Omega} |\nabla u|^2 - \mu^2 \int_{\Omega}|u|^2}{\int_{\Gamma} |u|^2}.
\end{align}

Next, $\mu^2$ could be an interior Dirichlet-Laplace eigenvalue. In this case the Dirichlet boundary value problem for the Helmholtz operator $-\Delta - \mu^2$ is not uniquely solvable. In what follows, we denote
\begin{equation}\label{eq:notationspecD}
	\text{spec}_D(\Omega):=\{\lambda, \quad \text{an interior Dirichlet eigenvalue of }\, -\Delta\}.
\end{equation}
We say the wavenumber $\mu_D$ is an {\it exceptional value} for {\sc Problem I} \eqref{eq:steklov-helmholtz} iff $\mu_D^2 \in \text{spec}_D(\Omega)$.

We see that if $\mu=\mu_D$ then the eigenvalue  {\sc Problem I} \eqref{eq:steklov-helmholtz} is not meaningful as stated, in the sense that the solution operator of the related boundary value problem is not uniquely solvable, and one must interpret the eigenvalue problem with care. Computationally, the conditioning of the related discrete problem degenerates as $\mu^2 $ gets close to $\mu^2_D\in \text{spec}_D(\Omega)$. In fact, as $\mu^2 \rightarrow  \mu_D^2$ (a Dirichlet eigenvalue of $-\Delta$), the corresponding Steklov-Helmholtz eigenvalue $\sigma \rightarrow - \infty$, with the same multiplicity  (see, eg., Section 7.4 in \cite{Levitin23}). 

These issues can be illustrated by examining {\sc Problem I} when $\Omega$ is the unit disk. In this case, the Steklov-Helmholtz eigenvalues are $ \sigma_k = \mu \frac{J_n'(\mu)}{J_n(\mu)}$ and the eigenfunctions are $J_n(\mu r)\exp(in\theta)$, where $J_n(t)$ is the Bessel function of order $n$. We note the $kth$ eigenvalue does not necessarily correspond to $k=n$, while the oscillations in the angular variable will depend on $n$. This effect is observed in \autoref{fig:circlen500mu71k2real}, where {the second ($k=2$) Steklov-Helmholtz eigenmode on the disk corresponds to $n = 1$ for $\mu = 7.015$ and  $n=4$ for $\mu = 7.1$}. In contrast, a second Steklov-Laplace eigenmode, corresponding to $\sigma_2^L= 1$, oscillates as $\exp(i \theta)$ in the angular variable. On the disk, the first eigenmode of the Steklov-Laplace problem has multiplicity 1 and the rest are 2. The multiplicity of the first Steklov-Helmholtz eigenmode need not be 1.  We refer to Section 1.4.2 in \cite{kshitijMS} for a discussion of this phenomenon.

\autoref{fig:circlen500mu71k2real}, also demonstrates the behaviour of Steklov eigenmodes when  $\mu$ is close to an exceptional wavenumber $\mu_D$ i.e. $\mu^2_D \in \text{spec}_D(\Omega)$. As we can see, both the eigenvalues and the eigenmodes depend on whether $\mu^2>\mu_D^2$ or not. As $\mu\rightarrow \mu_D^{-}$, $\sigma(\mu) \rightarrow -\infty$. The corresponding eigenmode appears close to a Dirichlet eigenmode for the Laplacian. 
\begin{figure}
	\centering
	\includegraphics[width=0.5\linewidth]{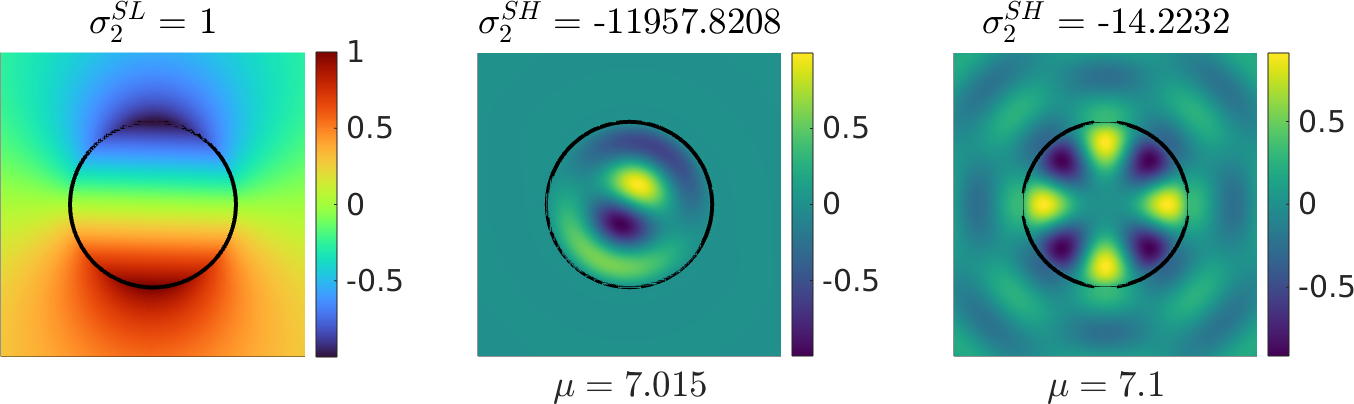}
	\caption{ (L) 2nd Steklov-Laplace eigenmode on the unit disk. (M) (Real part of) 2nd Steklov-Helmholtz eigenmode for $\mu =7.015$. Here  $\sigma=-11957.8208$ on the unit disk. (R) 2nd Steklov-Helmholtz  eigenmode for $\mu=7.1$, with $\sigma=-14.2232$.  We note that $7.0156...$ is an exceptional value.}
	\label{fig:circlen500mu71k2real}
\end{figure}

Most works in the literature concerning discretizations of {\sc Problem I} {\it require} that the wavenumber $\mu$ is non-exceptional.    By examining the Dirichlet-to-Neumann map for the Helmholtz operator, we are able to extend the definition of the Steklov-Helmholtz spectrum even for exceptional wavenumbers. This is discussed in some detail in Section 2; interested readers are also pointed to \cite{Levitin23,DNsemi}.

{\sc Problem Ia:} {\it Let $\mu\in \bb R$ be fixed and $\Omega$ be a bounded Lipschitz domain. Then,
\begin{enumerate}
\item if $\mu$ is not an exceptional wavenumber, find a function $u \in H^{1}(\Omega)$ and an associated eigenvalue $\sigma \in \mathbb{R}$ so that \autoref{eq:steklov-helmholtz} is satisfied.
\item if $\mu$ is an exceptional wavenumber, find a function $u \in H^{3/2}(\Omega)$ and an associated eigenvalue $\sigma \in \mathbb{R}$ so that \autoref{eq:steklov-helmholtz} is satisfied.
\end{enumerate}
}

We record the variational description of eigenvalues in {\sc Problem Ia} below. 

{\sc Problem Ia':} {\it Let $\mu\in\bb R$ be fixed and $\Omega$ be a bounded Lipschitz domain. For $k=1,2,...$, the $kth$ Steklov eigenvalue of $(-\Delta-\mu^2)$ is given by the following min-max characterization.
\begin{enumerate}
\item If $\mu$ is not an exceptional wavenumber,
\begin{equation}\label{eq:RRversion1}
	\sigma_{k}:= \min_{V_k \subset H_\mu(\Omega), \text{dim}(V_k)=k} \, \max_{w\in V_k, w \not=0} \mathcal{R}[\mu,w],
\end{equation}
where, $\mathcal{R}[\mu,w]$ is as defined in \autoref{eq:RQ} and $H_\mu(\Omega)\subset H^1(\Omega)$ is defined by 
$$ H_\mu(\Omega):=\{w \in H^1(\Omega)\vert   - \Delta w =\mu^2 w\}.$$

\item If $\mu = \mu_D$ is an exceptional wavenumber,
\begin{equation}\label{eq:RRversion2}
	\sigma_{k}:= \min_{V_k \subset \tilde H_\mu(\Omega), \text{dim}(V_k)=k} \, \max_{w\in V_k, w \not=0} \mathcal{R}[\mu,w].
\end{equation} 
where, $ \tilde H_\mu(\Omega) $ is the $L^2$-orthogonal complement of the space of the Dirichlet eigenfunctions at eigenvalue $\mu^2$,
 $$\tilde H_\mu(\Omega):=\{w \in H^{3/2}(\Omega)\vert   - \Delta w =\mu^2 w, \, w\vert_\Gamma =0\}.$$
\end{enumerate}
}

In practice, we may not know the exceptional values $\mu_D$ {\it a priori;} any successful computational strategy for \eqref{eq:steklov-helmholtz} must account for these. For computational efficiency, we should not independently require the solution of the Dirichlet eigenvalue problem  so as to avoid exceptional wave numbers. 

Computational strategies for Steklov problems \-- both for the Laplace and the Helmholtz operator \-- have been the subject of intense study in recent years. We briefly review some of the literature concerning finite element, spectral methods and integral-operator based approaches for the Steklov-Helmholtz problem \eqref{eq:steklov-helmholtz}.

The Steklov-Helmholtz problem for the operator $(-\Delta - \mu^2 n(x))$ given wavenumber $\mu \in \mathbb{R}$ and refractive index $n(x)$ is closely linked to  inverse problems in wave scattering.  
Here, the refractive index $n$ is allowed to be complex-valued; if $Im(n)\not=0$ the medium is absorbing. In this setting, the desired eigenvalues $\sigma$ may be complex, and require careful treatment. The inverse scattering problem is to retrieve $n(x)$ given suitable data, and one approach is to use the Steklov-Helmholtz spectrum. This connection has motivated considerable work on this (potentially non self-adjoint) version of \eqref{eq:steklov-helmholtz}, and in the literature is sometimes termed {\it the non self-adjoint} problem. An important work exploring this connection is \cite{Cakoni16}, in which a finite element approach was used to discretize the resultant eigenvalue problem. 
The location of complex eigenvalues of a non-Hermitian generalized eigenvalue system is numerically challenging, and one successful approach is via Beyn's method.  In \cite{SunTurner} a combination of a finite element discretization, and the spectral indicator method (a refinement of the recursive integral method of \cite{Huang16}) is used. Some other finite element approaches include those of 
\cite{Li22},\cite{Xie23},\cite{Zhang19},\cite{Zhang20}.

Approaches which use suitable orthogonal global bases include \cite{Wu17,Ren22, Harris21}. In \cite{Harris21}, a novel basis using Neumann-Laplace eigenfunctions are used to compute the Steklov spectrum, yielding a spectral approach. This work is within the context of inverse scattering problems, where the refractive index is allowed to be complex.

Amongst boundary-integral approaches we highlight \cite{Ma22}, in which a direct approach \--- the eigenmode being solved for is the Dirichlet trace of the eigenfunction 
\--- is used in combination with the spectral indicator method. Here, it is assumed that $\mu$ is not an exceptional value. In our work, we employ an indirect layer approach, and aim to provide a wavenumber-robust strategy.

Our work is motivated by applications from spectral geometry rather than inverse scattering, and we focus on real refractive indices. Our aim is to provide a high-accuracy, efficient and wavenumber-robust approach, which can be used within the context of shape optimization for eigenvalues. The eigenvalue problem \eqref{eq:steklov-helmholtz} is not well-defined as stated when $\mu^2\in \text{spec}_D(\Omega)$. We exploit an identification via a (carefully defined) Dirichlet-to-Neumann map \cite{Levitin23} to nonetheless capture the spectrum in this case.  

We briefly describe quadrature rules for both smooth and piecewise smooth boundaries in \autoref{Section:quadrature}, and refer to \cite{ColtonKress13} for details. Our focus in this paper is the suitable formulation of an eigenvalue problem, and addressing ill-conditioning arising from the presence of exceptional wavenumbers; any of several quadrature approaches can be used to build the discrete matrices involved. After describing the proposed algorithm, we use it to numerically investigate some questions concerning the spectral geometry of the Steklov-Helmholtz eigenvalue problem.

We finally note that when $\mu = \iota \gamma, \gamma \in \mathbb{R}$ is strictly imaginary, $(-\Delta +\gamma^2)$ is a strongly elliptic operator. Noteworthy work includes a Boundary Element Method based approach in \cite{Tang98}, an approach based on mechanical quadrature in \cite{Huang04} and recently {\cite{denis25} where asymptotic behaviour of exterior eigenvalues in validated by an Finite Element Method approach}. In this case, the spectrum is strictly positive, and the issues described above do not arise. Nonetheless, high-accuracy discretizations in this setting are of interest and relevance for the present work.

\section{The Steklov-Helmholtz EVP and reformulations}

Our approach in this paper is to use integral operators to reformulate the Steklov-Helmholtz eigenvalue problem, and then discretize the resultant (new) problem. Concretely, we use integral operators to define a problem that is isospectral to the desired EVP; the eigenfunctions of the latter can be recovered via postprocessing. In this paper we work on planar domains with Lipschitz boundaries $\Gamma$, for which the quadratures are well-known and standard.

The use of layer potentials to locate eigenvalues is not new, and is similar to the ideas in \cite{Akhmetgaliyev,Cakoni17, Ma22}. A different approach is proposed in \cite{barnett2014fast}. Here, a generalization of the Cayley transformation is used in combination with layer potentials to examine the Neumann-to-Dirichlet map on star-shaped domains; the goal in~\cite{barnett2014fast} is to efficiently compute Dirichlet eigenvalues to very high accuracy. The authors remark their approach can also be used to retrieve Steklov eigenvalues for the Helmholtz operator. 
The work \cite{Ma22}, based on a direct approach, is most closely related to the topic of our paper. In contrast to this previous work, we use an {\it indirect} approach, and obtain a related eigenproblem involving the adjoint of the double layer potential. Further, we seek an approach which works for {\it all} real wavenumbers $\mu$, avoids issues of poor conditioning, and works for a wide class of domain shapes (including on non-convex domains and curvilinear polygons). 

We  define  the familiar single layer operator $\mathcal{S}_\mu:H^{-1/2}(\Gamma)\rightarrow H^{1/2}(\Gamma)$ and the adjoint of the double layer operator, $\mathcal{K'}_\mu:H^{-1/2}(\Gamma)\rightarrow H^{-1/2}(\Gamma)$ as

\begin{eqnarray*} 
	\mathcal{S}_\mu[\phi](x)&:=&  \frac{\iota}{4}\int_{\Gamma}  H_0^1(\mu|x-y|)\phi(y) {\d} s(y)\quad x\in \Gamma \\
	\mathcal{K'}_\mu[\phi](x)&:=&\frac{\iota}{4}\int_{\Gamma}  \frac{\partial H_0^1(\mu|x-y|)}{\partial n(x)}\phi(y){\d}s(y),\quad x\in \Gamma, \end{eqnarray*}
	where $H_0^1(x)$ is the Hankel function of the first kind, order zero.


	\subsection{{Layer potential formulation with $\mu$ non-exceptional}} Let us first consider the case when  $\mu^2\not \in \text{spec}_D(\Omega)$. We seek solutions of {\sc Problem IA} in~\eqref{eq:steklov-helmholtz} using an indirect approach via a {\it single layer ansatz,} similar to the approach in \cite{Akhmetgaliyev}. That is, we assume eigenfunctions $u$ of \eqref{eq:steklov-helmholtz} satisfy
	\begin{equation}\label{eq:ansatz} u(x) = \frac{\iota}{4}\int_{\Gamma}  H_0^1(\mu|x-y|)\phi(y) {\d}s(y) :=\widetilde{S}_{\mu} \phi \quad \quad x\in \Omega, \end{equation}  for some density $\phi \in H^{-1/2}(\Gamma)$. Here, $ \widetilde{S}_{\mu}: H^{-1/2}(\Gamma)\rightarrow H^1(\Omega)$ is the {\it single layer potential,} and is related to $\mathcal{S}_\mu$ through the trace operator, $\mathcal{S}_\mu = Tr\circ \widetilde{S}_{\mu}$. By construction, $u = \widetilde{S}_\mu \phi$ will satisfy the Helmholtz equation $ -\Delta u - \mu^2 u=0$ in $\Omega$.
	On the boundary, we can use the well-known jump relations and \eqref{eq:steklov-helmholtz} to obtain the eigenvalue problem for the density $\phi$:
	
	{\sc Problem II:} {\it Suppose $\mu^2 \not\in \text{spec}_D(\Omega).$ Find $\sigma \in \mathbb{R}$ and nonzero $\phi \in H^{-1/2}(\Gamma)$ so that} 
	\begin{align}  \label{eq:steklovbie} 
		(\mathcal{K'}_\mu + \tfrac{1}{2}) \phi(x) = \sigma \mathcal{S}_\mu\phi(x),\qquad \forall x \in \Gamma. \end{align}
	
	If $\mu$ is not an exceptional value then {\sc Problem Ia} \eqref{eq:steklov-helmholtz} and {\sc Problem II} \eqref{eq:steklovbie} are isospectral in the following sense:
	$$ (\mathcal{K'}_\mu + \tfrac{1}{2}) \phi(x) = \sigma \mathcal{S}_\mu\phi(x) \quad \Leftrightarrow \quad  -\Delta u - \mu^2 u =0,  \quad 	\frac{\partial u}{\partial n} =\sigma \T u \quad \text{on}\,\, \Gamma.$$
	Provided $\mu^2_D \not\in \text{spec}_D(\Omega),$  a computational approach is to discretize \eqref{eq:steklovbie} and locate the eigenpairs $(\phi, \sigma)$. This approach allows us to achieve  high-accuracy approximations of the Steklov-Helmholtz eigenvalues for a wide range of domains.  The desired eigenfunctions $u$ of \eqref{eq:steklov-helmholtz} can then be easily reconstructed from the ansatz \eqref{eq:ansatz} using a quadrature. 
	
	We observe that for wavenumbers $\mu$ {\it close} to exceptional ones, the conditioning of numerical discretization approaches for {\sc{Problem Ia}} and {\sc{Problem II}} worsens. Discretizations of $\mathcal{S}_\mu$ become nearly singular.  One of the main challenges we focus on, therefore, is removing the restriction $\mu^2_D \not\in \text{spec}_D(\Omega);$ this in turn allows us to provide a wavenumber-robust discretization. 
	
	\subsection{{The Dirichlet-to-Neumann map formulation with $\mu \in \mathbb{R}$}}
	As mentioned in the introduction, the Steklov-Helmholtz eigenvalue {\sc{Problem Ia}} in  \eqref{eq:steklov-helmholtz} must be recast for exceptional wavenumbers $\mu_D$. We seek a problem which is isospectral to {\sc{Problem Ia}} and {\sc{Problem II}} when $\mu$ is not an exceptional value, and remains well-defined even when $\mu^2 \in \text{spec}_D(\Omega)$. We note that another interpretation of the Steklov-Helmholtz problem is through the associated eigenvalue problem for the {\it Dirichlet-to-Neumann map.} In this subsection we review the key ideas, following very closely the treatment \cite{Levitin23}. 
	
	Suppose $\mu$ is NOT exceptional. Define $\mathcal{E}_\mu$ as the Helmholtz extension: given $f\in H^{1/2}(\Gamma)$, the map $
	\mathcal{E}_\mu:H^{1/2}(\Gamma) \rightarrow H^1(\Omega)$ is defined by $
	\mathcal{E}_\mu f:=W$ where $W$ is the weak solution of  
	\begin{align*}
		- \Delta W - \mu^2 W =0 \,\, \text{in} \, \Omega, \qquad  W = f \,\, \text{on} \,\, \Gamma.
	\end{align*}
	This extension map is well-defined on $H^{1/2}(\Gamma)$.  The Dirichlet-to-Neumann map associated with the Helmholtz operator $ \mathcal{\widetilde{D}}_\mu: H^{1/2}(\Gamma) \rightarrow  H^{-1/2}(\Gamma)$ can then be  defined by:
	\begin{align*}
		\mathcal{\widetilde{D}}_\mu(u) =   \left( \frac{\partial}{\partial n} (\mathcal{E_\mu}u) \right)_{|\Gamma}.
	\end{align*}
	As is standard, we interpret $(\partial_n U)\vert_\Gamma$ as
	\begin{align} \label{def:normal-bdry}
		\int_\Gamma (\partial_n U)\vert_\Gamma \, \T (v) \, {\d}s = \int_{\Omega} \nabla U \cdot \nabla v \, {\d}x -\mu^2 \int_{\Omega}  v  U \, {\d}x,
	\end{align}
	for any $v \in H^1(\Omega)$ and any $U\in H^{3/2}(\Omega)$ which  satisfies the Helmholtz equation. For smooth $\Gamma$,  the map $\mathcal{\widetilde{D}}_\mu$ is an elliptic, self-adjoint, and order 1 pseudo-differential operator. 
	
	If $\mu^2 \not \in \text{spec}_D(\Omega)$, then the spectrum of $\mathcal{\widetilde{D}}_\mu$ is the same as the Steklov-Helmholtz spectrum of Problem I, though the eigenfunctions are different.
	Therefore, 
	$$ \mathcal{\widetilde D}_\mu f= \sigma f \Leftrightarrow - \Delta u -\mu^2 u =0 \quad \mbox{in} \,\, \Omega, \quad \partial_n u =\sigma\, \T(u) \,\, \mbox{on} \,\, \Gamma.$$
	In this case, the single layer operator $\mathcal{S}_\mu$ is invertible, and we can write $\mathcal{\widetilde{D}}_\mu = (\mathcal{K'}_\mu + \frac{1}{2}) {\mathcal{S}}_\mu^{-1}.$  
	However, if $\mu$ is an exceptional value, then $\mathcal{\widetilde{D}}_\mu$ is not well-defined on $H^{1/2}(\Gamma)$, and one must restrict the domain by avoiding the densities associated with Dirichlet eigenfunctions of the Laplacian. In this case, the extension $\mathcal{E}_\mu$ may be multi-valued. 
	
	We consider the set of Cauchy data of the solutions of the problem $$ - \Delta f = \mu^2 f, \quad \T(f) = \phi.$$ We note that for all $\mu \in \mathbb{R}$, the sets
	\begin{align*}
		\mathcal{F}(\mu):=\left\{\left(\T(w), \frac{\partial}{\partial n}w \right) \in H^{1/2}(\Gamma)\times H^{-1/2}(\Gamma) \vert w \in H^1(\Omega), -\Delta w = \mu^2 w\right\}\\
			\mathcal{G}(\mu):=\left\{\left( \frac{\partial}{\partial n}w, \T(w) \right) \in H^{-1/2}(\Gamma)\times H^{1/2}(\Gamma) \vert w \in H^1(\Omega), -\Delta w = \mu^2 w\right\}\\	
	\end{align*}
	can be viewed as  graphs of (potentially multi-valued) operators. In the case when $\mu$ is not an exceptional value, $\mathcal{F}(\mu)$ is the graph of the Dirichlet-to-Neumann map $\mathcal{\widetilde{D}}_\mu$. Working with the {\it linear relations} $\mathcal{F}(\mu)$ and $\mathcal{G}(\mu)$ allows us to extend the definition of the Dirichlet-to-Neumann map for all $\mu \in \mathbb{R}$. This is the approach taken in \cite{berhndt,Levitin23, arendt} and allows for the definition of the domain of the linear relation $\mathcal{F}$ as
	$$ \text{dom }{\mathcal F}(\mu):= \{ g \in H^{1/2}(\Gamma)\vert \left(g,h\right) \in \mathcal{F}(\mu) \, \text{for some }\, h\in H^{-1/2}(\Gamma)\};$$
	one can also characterize the set
		$$ \text{mul}\,{\mathcal F}(\mu):= \{ h \in H^{-1/2}(\Gamma)\vert \left(0,h\right) \in \mathcal{F}(\mu)\}.$$
	The linear relation $\mathcal{F}(\mu)$ is the graph of an operator iff $\text{mul}{\mathcal F}=\{0\}$.
	By examining the restrictions $F(\mu) := \mathcal{F}(\mu) \cap (L^2(\Gamma)\times L^2(\Gamma))$ again from the viewpoint of linear relations, and exploiting elliptic regularity results, it was shown in \cite{berhndt} that $F(\mu)$ is a self-adjoint relation in $L^2(\Gamma)$ with finitely many negative eigenvalues, and the operator part is an unbounded operator with discrete spectrum. This is true for {\it all} $\mu \in \mathbb{R}$, Theorem 5.10 \cite{berhndt}. We use these ideas and follow the treatment in Chapter 7.4 in \cite{Levitin23} to guide our algorithm development. Suppose $\mu=\mu_D$ is an exceptional value for $\Omega$, i.e., $\mu_D^2\in \text{spec}_D(\Omega)$.  We define the space
	\begin{align}
		\mathcal{H}_{\mu_D}:=\left\{ \frac{\partial u_D}{\partial n} \Big{\vert} \, \, -\Delta u_D = \mu_D^2 u_D \,\, \text{in}\,\, \Omega,  \,\, u_D=0\,\, \text{on} \,\, \Gamma \right\}.
	\end{align} That is, $\mathcal{H}_{\mu_D} (\equiv \text{mul} \, \mathcal{F}(\mu_D))$ is the set of Neumann traces of the Laplace-Dirichlet eigenfunctions $u_D$ corresponding to the interior Dirichlet eigenvalue $\mu_D^2$. The space is finite-dimensional since Dirichlet eigenvalues have finite multiplicity. 
	
	
	The Dirichlet-to-Neumann operator can now be defined for {\it all} real wavenumbers $\mu$, as a map $ \mathcal{D}_\mu:\text{dom}(\mathcal{D}_\mu) \rightarrow  H^{-1/2}(\Gamma)$ with:
	\begin{align} \label{def:DtN-all-mu}
		\mathcal{D}_\mu(u) 
		= \begin{cases}   
			\left(\frac{\partial}{\partial n} (\mathcal{E_\mu}u)\right)_{\Gamma},\quad &\mu^2 \notin \text{spec}_D(\Omega),\\[4pt]
			\Pi_{\mathcal{H}_\mu^\perp}\left( \frac{\partial}{\partial n} (\mathcal{E_\mu}u) \right)_{\Gamma}, \quad &\mu^2 \in \text{spec}_D(\Omega).
		\end{cases}
	\end{align}
	Note that the projection $\Pi_{\mathcal{H}_\mu^\perp}$ removes the potential multivaluedness of $\mathcal{E}_\mu$ when $\mu$ is exceptional. Again from  \cite{Levitin23}, the domain of $\mathcal{D}_\mu$ is 
	\begin{align}
		\text{dom}(\mathcal{D}_\mu):=\begin{cases}
			H^{1/2}(\Gamma), & \mu^2  \not \in\text{spec}_D(\Omega), \\
			H^{1}(\Gamma)\cap \mathcal{H}_\mu^{\perp}, &\mu^2\in\text{spec}_D(\Omega), 
	\end{cases}\end{align}
	where the orthogonality is in $L^2(\Gamma)$. We also define the range as
	\begin{align}
		\text{range}(\mathcal{D}_\mu):=\begin{cases}
			H^{-1/2}(\Gamma), & \, \mu^2  \not \in\text{spec}_D(\Omega), \\
			L^2(\Gamma), &\, \mu^2\in\text{spec}_D(\Omega).
		\end{cases}
	\end{align}
	Note that by restricting to $\mathcal{H}_\mu^\perp$, $\mathcal{D}_\mu$ is a well-defined mapping on $\text{dom}(\mathcal{D}_\mu)$ for all $\mu \in \mathbb{R}$.

	With this definition, we can state the eigenvalue problem for the Dirichlet-to-Neumann map for the Helmholtz operator, with no restrictions on $\mu$.
	
	\noindent{\sc Problem III:} {\it Let $\mu\in \mathbb{R}$ be given. Find eigenfunctions $\psi \in \text{dom}(\mathcal{D}_\mu$) and eigenvalues $\sigma \in \mathbb{R}$ so that}
	\begin{align}
		\label{eq:dtncareful}
		\mathcal{D}_\mu \psi = \sigma \psi.
	\end{align}
	
	\begin{lemma}
		Let $\mathcal{D}_\mu$ be the Dirichlet-to-Neumann map defined in~\eqref{def:DtN-all-mu} for all $\mu \in \mathbb{R}$. Then  {\sc{Problem III}} is isospectral to \sc{Problem Ia}.  
	\end{lemma}
	\begin{proof}
		We only need show the equivalence for $\mu$ being exceptional. First, suppose $(\sigma, \psi)$ is an eigenpair for {\sc{Problem III}}. Among the multi-values of $\mathcal{E}_\mu \psi$, we choose the one such that $\left(\frac{\partial}{\partial n}\mathcal{E}_\mu \psi \right)_\Gamma \in \mathcal{H}_\mu^\perp$. Then
		\begin{align*} 
			\left(\frac{\partial}{\partial n} (\mathcal{E_\mu}\psi) \right)_{|\Gamma}
			= \Pi_{\mathcal{H}_\mu^\perp}\left(\frac{\partial}{\partial n} (\mathcal{E_\mu}\psi) \right)_{|\Gamma}
			= \sigma \psi.
			\end{align*}
		Therefore, $(\sigma, \mathcal{E}_\mu \psi)$ is an eigenpair of {\sc{Problem Ia}}. 
		
		Next, suppose that $(\sigma, u)$ is any eigenpair of {\sc{Problem Ia}}, such that $\T(u) \not=0$. (That is, we consider an eigenpair of {\sc{Problem Ia}} corresponding to a finite value of $\sigma$.) Suppose $w$ is a Dirichlet eigenfunction of the Laplacian (Dirichlet eigenvalue $\mu^2$). Then, since $u\in H^1(\Omega)$ is a Steklov-Helmholtz eigenfunction, and since $w \in H^1_0(\Omega)$ solves $ -\Delta w= \mu^2 w$, we get
		$$ \mu^2\int_\Omega uw\,{\d}\,\Omega = \int_{\Omega} \nabla u \cdot \nabla w \, {\d}\Omega, \quad  \mu^2\int_\Omega uw\,{\d}\,\Omega = \mu^2\int_{\Omega} u \,w \, {\d}\Omega + \int_{\Gamma} u \frac{\partial w}{\partial n}\, {\d}s. $$ From this it follows that $\int_{\Gamma} u \frac{\partial w}{\partial n}\, {\d}s =0$, that is, $\T(u) \in \mathcal{H}_\mu^{\perp}$. 
		From elliptic regularity, it can be seen that $w \in H^{3/2}(\Omega)$,  (eg.\cite{mitrea}, \cite{berhndt} Lemma 5.1 and Theorem 5.2). 
		
		Now let $\bar{u}=u+w \in H^{3/2}(\Omega)$. Since $\sigma \T(u) = \frac{\partial u}{\partial n}$, we have
	$$	\Pi_{\mathcal{H}_\mu^\perp}
		\left(\frac{\partial \bar u}{\partial n} \Big\vert_\Gamma \right)
		= \Pi_{\mathcal{H}_\mu^\perp}
		\left(\sigma \T(\bar u)\right)
		= \sigma \T(\bar u),$$ and hence $(\sigma, \T(\bar{u}))$ is an eigenpair of {\sc Problem III}.
	\end{proof}
	
	We record a variational characterization of the eigenvalues of the Dirichlet-to-Neumann map, noting this is well-defined for {\it all} wavenumbers $\mu$:
	\smallskip
	
	\noindent {\sc Problem III':}{\it Let the wavenumber $\mu\in \mathbb{R}$ be fixed. The $kth$ eigenvalue of $\mathcal{D}_\mu$ is given by} 
	$$ \sigma_k(\mu)= \min_{X \subseteq \text{dom}{(\mathcal{D}_\mu)}, \, dim(X)=k} \max_{f\in X, f\not=0} \frac{\|\nabla \mathcal{E}_\mu f|^2_{L^2(\Omega)} - \mu^2\|\mathcal{E}_\mu f\|^2_{L^2(\Omega)}}{\|f\|^2_{L^2(\Gamma)}}.$$
	Once again, in analogy to the variational characterization \eqref{eq:RRversion2}, the admissible functions avoid a (finite-dimensional) set related to Dirichlet eigenfunctions of the Laplacian.  
	
	
	\subsection{{Layer potential formulation with $\mu \in \mathbb{R}$}}
	We would like to define the Dirichlet-to-Neumann map $\mathcal{D}_\mu: \text{dom}(\mathcal{D}_\mu) \to \text{range}(\mathcal{D}_\mu)$ for the self-adjoint operator $(-\Delta - \mu^2)$ in terms of boundary integral operators, (see eg. \cite{mclean}). {Recall} that if $\mu \in \mathbb{R}$ is not an exceptional value, then  
	\begin{equation}\label{eq:dtn_bie}
		\mathcal{D}_\mu =(\tfrac{1}{2}I + \mathcal{K'}_\mu)\circ \mathcal{S}^{-1}_{\mu}.
	\end{equation}
	However, if $\mu=\mu_D$ is exceptional, then \eqref{eq:dtn_bie} must be modified. First, the kernel of the single layer operator $\text{ker}{\mathcal S}_{\mu_D}$ is a subset of $L^2(\Gamma)$ due to regularity arguments (see, eg. Theorem 5.2 in \cite{berhndt}). Moreover, $u_D \in H^{3/2}(\Omega)$ is a Dirichlet eigenfunction of the Laplacian iff $u_D = \mathcal{\widetilde{S}}_{\mu_D} \phi_D$ 
	for some $\phi_D\in  \text{ker}{\mathcal S}_\mu$. 
	
	The layer-potential characterization of the space $\mathcal{H}_{\mu_D}$ is as follows: suppose we represent the interior Dirichlet eigenfunctions of $\Omega$ corresponding to $\mu_D^2$ in terms of the single layer operator such that
	\begin{align}
		u_D(x)= \int_{\Gamma}H^1_0(\mu_D(x-y)) \psi(y)\, \d s_y = \widetilde{\mathcal{S}}_{\mu_D} \psi. 	
	\end{align}  
	Then the trace of the normal derivative of $u_D$ can be written in terms of the adjoint of the double layer operator:
	\begin{align}
		\frac{\partial u_D }{\partial n} = (\mathcal{K'}_{\mu_D}
		+\tfrac{1}{2}\mathcal{I})\psi. 
	\end{align}
	Therefore, we arrive at the layer-potential characterization of the space $\mathcal{H}_{\mu_D}$:
	\begin{align}
		\mathcal{H}_{\mu_D}:=\{w \in L^2(\Gamma)\vert  w = (\mathcal{K'}_{\mu_D}+\tfrac{1}{2}\mathcal{I})\psi, \,\, \, \text{where} \,\, \psi \in \text{ker}{\mathcal S}_\mu \}.
	\end{align}
	
	We are led to a problem in terms of boundary integral operators:
	\smallskip
	
	\noindent {\sc Problem IV:} {\it Let $\mu\in \mathbb{R}$ be given. Find a nonzero density $\phi \in L^2(\Gamma)$, and $\sigma \in \mathbb{R}$ which satisfy the constrained eigenvalue problem:}
	\begin{subequations}
		\begin{align}
			{\Pi_{\mathcal{H}_\mu}^\perp}(\mathcal{K'}_\mu + \frac{1}{2})\phi &= \sigma \mathcal{S}_\mu\phi,
			\\ 
			(\phi,\psi)&=0,\quad  \forall \psi \in \text{ker}\mathcal{S}_\mu,
		\end{align}
	\end{subequations}
	where we constrain the domain of densities to those orthogonal to the (finite dimensional) kernel of $\mathcal{S}_\mu$ for a unique $\phi$.
	
	\begin{lemma}
		{\sc Problem IV} is isospectral to {\sc Problem III} for the (suitably defined) Dirichlet-to-Neumann map $\mathcal{D}_\mu$. 
	\end{lemma}
	
	\begin{proof}
		In case $\mu^2 \not \in \text{spec}_D(\Omega)$ the claim is clear.    
		Now suppose $\mu^2$ is a Dirichlet eigenvalue of the Laplacian, and $(\sigma, \phi)$ solves {\sc Problem IV}. The function $f:=\mathcal{S}_\mu \phi \in H^1(\Gamma)$ since $\phi \in L^2(\Gamma)$, and by construction, any Helmholtz extension of the boundary trace $f$ is of the form
		\begin{align*}
			\mathcal{E}_\mu f = \widetilde{S}_\mu \phi +w_D.
		\end{align*}
		By the definition of the Dirichlet-to-Neumann map $\mathcal{D}_\mu$,
		\begin{align*}
			\mathcal{D}_\mu f ={\Pi_{\mathcal{H}_\mu}^\perp} \left(\frac{\partial}{\partial n} \widetilde{S}_\mu \phi \right)
			={\Pi_{\mathcal{H}_\mu}^\perp} \left(\frac{1}{2}+\mathcal{K}'_\mu \right)\phi = \sigma \mathcal{S}_\mu \phi = \sigma f,
		\end{align*}
		showing that $(\sigma,f)$ solves {\sc Problem III}. If $(\sigma, u)$ solves {\sc Problem III}, then {\sc Problem IV} holds by letting $u = \mathcal{S}_\mu \phi$ and $\phi \in (ker\mathcal{S}_\mu)^\perp$.
	\end{proof}
	
	We note that the single layer operator $\mathcal{S}_\mu$ is invertible on bounded, smooth simply connected domains for $\mu\in \mathbb{R}$ except for a countable number of points (when $\mu^2$ is an interior Dirichlet eigenvalue of $-\Delta$). Additionally, $\mathcal{S}_0$ may fail to be invertible if the logarithmic capacity of $\Omega=1$. For the rest of this paper, if $\mu=0$ we replace $\mathcal{S}_0$ by the modified single layer operator $\mathcal{S}_0\leftarrow\tilde{\mathcal{S}}_0$ (see, eg. \cite{ColtonKress13}). With this change, $\mathcal{S}_{\mu}$ fails to be invertible only if $\mu = \mu_D$.

\section{A discretization approach for the Steklov-Helmholtz problem}
We now detail the procedure for computing the spectrum of {\sc Problem II} and {\sc Problem IV}. We propose a collocation-based approach, as opposed to a Galerkin approach via finite element or boundary element methods. We then describe a wavenumber-robust strategy to solve the resulting generalized eigenvalue problem.

\subsection{Quadrature} \label{Section:quadrature}

In the {\sc Problems} II and IV, the integral operators $\mathcal{K'}_\mu$ and $\mathcal{S}_\mu$ must be discretized. For smooth boundaries this is done via standard quadrature schemes first proposed by \cite{Kussmaul69} and \cite{Martensen63}. In the case of piecewise smooth domains, other sophisticated quadrature schemes are available, including \cite{Helsing08, Bruno01}, and others which include grading towards the corners, \cite{ColtonKress13, Klockner13}. This list is not exhaustive. We now recall some standard quadrature approaches used in the numerical results sections.

In case $\Gamma$ is smooth, we can parametrize it as
$\Gamma:=\{(x(t),y(t)),~ t\in [0,2\pi ] \}$, and the integrals in {\sc Problem II} \eqref{eq:steklovbie} can be replaced by integrals over $[0,2\pi]$. Noting the kernels in the single and double layer operators are  singular, we employ the technique of \cite{Kussmaul69} and \cite{Martensen63} of adding and subtracting  $\ln\left(4\sin^2\frac{t-\tau}{2}\right)$ from the kernels:

\begin{align*}
	\mathcal{S}_{\mu} \phi(x(t))\equiv 	\mathcal{S}_{\mu} \psi(t) &=  \int_0^{2\pi} \left[M_1(t,\tau)\ln\left(4\sin^2\frac{t-\tau}{2}\right)+M_2(t,\tau)\right]\psi(\tau)\d \tau,\\
	\mathcal{K'}_{\mu} \phi(x(t))\equiv 	\mathcal{K'}_{\mu} \psi(t)=&\int_{0}^{2\pi} \left[L_1(t,\tau)\ln\left(4\sin^2\frac{t-\tau}{2}\right)+L_2(t,\tau)\right] \psi(\tau)\d \tau.
\end{align*}

Now we have integrals with logarithmic kernels of the form,
\[
\int_0^{2\pi}f(t,\tau)\ln\left(4\sin^2\frac{t-\tau}{2}\right)+g(t,\tau)\d \tau,
\]
where $f$ and $g$ are smooth and $2\pi$ periodic. The smooth integrals are evaluated using the Trapezoidal rule.  
The approximation of the logarithmically-singular integral is given  by \cite{ColtonKress13}
\begin{eqnarray*}
	&\int_0^{2\pi}f(t,\tau)\ln\left(4\sin^2\frac{t-\tau}{2}\right) \d \tau \approx \sum_{j=0}^{2N-1}R^N(t,\tau_j)f(t,\tau_j), \text{ where}\\
	&R^N(t,\tau_j) := -\frac{2\pi}{N}\sum_{m=1}^{N-1}\frac{\cos m(t-\tau_j)}{m}-\pi\frac{\cos N(t-\tau_j)}{N^2}, \quad \tau_j = \frac{2\pi j}{N}.
	\\
\end{eqnarray*}

In the case of domains $\Omega$ with {$M$} Lipschitz boundaries $\Gamma = \cup_{L=1}^M \Gamma_L$ where each boundary part $\Gamma_L$ is smooth, we can introduce polynomially-graded meshes on each $\Gamma_L$. Suppose that
$$ \int_{\Gamma_L} f \, ds = \int_a^b f(t)\, dt$$ and that $f$ has singularities at the endpoints.
On the interval $t\in (a,b]$, the  polynomial change of variables $t=w(s)$ is designed in a way that the gradient $w'(s)$ at the corners vanishes up to order $p \in \mathbb{N}$. In particular, from \cite{ColtonKress13},
\begin{equation}\label{eqn:cov}
	\begin{split}
		&w(s) = a+(b-a) \frac{[v(s)]^p}{[v(s)]^p+[v(2\pi-s)]^p}, \quad 0\leq s\leq 2\pi, \ \\
		&v(s) = \left(\frac{1}{p}-\frac{1}{2}\right)\left(\frac{\pi-s}{\pi}\right)^3+\frac{s-\pi}{p\pi}+\frac{1}{2},~p\geq 2.
	\end{split}
\end{equation}

This form of polynomial grading \cite{ColtonKress13} (see Section 3.5) and \cite{Akhmetgaliyev} allows for about half the points on the segment to be accumulated towards the end points.  In practice we observe the conditioning of the matrices deteriorates for $p>6$.  We remark that a generalization of this quadrature due to \cite{jeon93}, based on rational changes of variable, was also tested as part of this eigenvalue discretization strategy; however, the method did not appear robust to choices of grading parameters in this setting. 

Graded meshes can also be used for problems with mixed Steklov and Neumann data on polygonal domains. On smooth domains with such junctions, one can achieve even higher accuracy by incorporating detailed information about the asymptotic behaviour of the eigendensities, similarly to the approach in \cite{Akhmetgaliyev}.

 Other quadrature approaches could also be implemented. We show, for instance, the performance of ChunkIE \cite{Askham_chunkIE_a_MATLAB_2024}, a 2D integral equation package which uses the RCIP quadature approach \cite{helsing}. The advantage of this approach is that the discrete single and double-layer matrices are better conditioned than those from a standard graded-mesh approach.

Using the chosen quadrature rules, we are led to the following generalized discrete eigenvalue problem for {\sc Problem II:}

{\sc Problem IID:} {\it Suppose $\mathtt{S}_\mu$ is nonsingular, ie, $\mu$ is not an exceptional wavenumber. Find nonzero ${\tt p} \in {\mathbb{C}}^{N}$ and $\sigma \in \mathbb{R}$ so that}
\begin{eqnarray}\label{eq:2D}
		\mathtt{(K_\mu +\tfrac{1}{2}I)p} = \sigma \mathtt{S_\mu p}.
\end{eqnarray} 
Variants of {\sc Problem IID} are considered in other related works, including \cite{Ma22}.

Similarly, the discrete version of {\sc Problem IV} is given as:

{\sc Problem IVD:} {\it Let  $\mu$ be any real wavenumber. Find nonzero $\mathtt{p}\in {\mathbb{C}}^{N}$ and $\sigma \in \mathbb{R}$ so that}
\begin{eqnarray}\label{eq:4D}
	\mathtt{(K_\mu +\tfrac{1}{2}I)}\mathtt{p} &= &\sigma \mathtt{S}_\mu \mathtt{p},\qquad 
	\mathtt{p}^T  \mathtt{q} =0, \quad \forall \mathtt{q}\in \mathtt{ker(S_\mu)}.
\end{eqnarray} 	
	
 We note these quadrature schemes will yield high-accuracy approximations for the operators themselves; if we wish to compute the volumetric eigenfunctions $u$ of {\sc Problem I}, we can postprocess these using \eqref{eq:ansatz}. While accuracy may degenerate as we evaluate closer to the boundary $\Gamma$, these integrals can also be achieved with high accuracy (for instance, \cite{barnett14,ludvig}).

In the rest of the paper, we restrict ourselves to either Nystr\"om or graded-mesh quadratures, and focus instead on aspects of eigenvalue approximation.
\subsection{The BIO-MOD approach}
Upon discretization, we obtain generalized eigenvalue problems \eqref{eq:2D} and \eqref{eq:4D}. The matrices $\mathtt{(K_\mu +\frac{1}{2}I)}$ and $\mathtt{S}_\mu$ are not necessarily Hermitian. 
We  observe (trivially) that $\mathtt{S_\mu}$ has a non-empty kernel if $\mu^{{2}}$ is an interior Dirichlet eigenvalue (ie, $\mu$ is an exceptional wavenumber); $\mathtt{(K_\mu +\frac{1}{2}I)}$ will have a non-empty kernel if $\mu^2$ is an interior Neumann eigenvalue.

If $\mu_D^2\not \in \text{spec}_D(\Omega)$, the matrix { $\mathtt S_\mu$} is non-singular. The discrete eigenvalues $\{\sigma_{i,N}\}_{i=1}^{N}$ of {\sc Problem IID} can be obtained using, for instance, {\tt eig} in {\tt MatLab}. We term this approach the {\it BIO approach}. This approach leads to high-accuracy eigenvalue approximations if $\mu$ is complex or far from an exceptional wavenumber. 
However, if $\mu_D^2 \in \text{spec}_D(\Omega)$, then the performance of both Arnoldi and the QZ algorithms deteriorate. The need to develop  strategies for computing the spectrum of the Dirichlet-to-Neumann map which work efficiently, accurately and stably for all real wavenumbers was originally motivated by applications in shape optimization, described later. 

  If $\mu^2_D\in \text{spec}_D(\Omega)$ is a Dirichlet eigenvalue of multiplicity $\ell$, then $\sigma_{i}=-\infty,i=1,\cdots,\ell$. The $\ell$ smallest singular values of  $\mathtt S_\mu$ are zero, and the BIO approach based on solving generalized eigenvalue problem in \autoref{eq:2D} leads to poor accuracy.  
  
 Motivated by the formulation of {\sc Problem IV} and the discrete variant {\sc Problem IVD}, we propose instead a wavenumber robust modification of the BIO approach, termed {\it BIO-MOD}. 
 
 Write the reduced SVD of the discrete single layer \begin{equation}
\mathtt{S_\mu=\hat{U}_\mu\hat{\Sigma}(\mu)\hat{V}_\mu^*},
\end{equation} where $\mathtt{\hat\Sigma}(\mu)$ is a square real diagonal matrix of size $r = N-\ell$ with non-zero singular values. By construction, the columns of ${\tt \hat{V}_\mu}$ provide an orthogonal basis of $\mathtt{(null(S_\mu))^\perp}.$ We can write the discrete solution {\tt p} of {\sc Problem IVD} as a linear combination of the $N-\ell$ columns of $V$, 
$$ \mathtt p = \mathtt{V_\mu P}, \quad \mathtt{P} \in {\mathbb{C}}^{N-\ell}.$$ 
An direct computation shows
\begin{equation}
\mathtt{(K_\mu +\tfrac{1}{2}I)}\mathtt{p} = \sigma \mathtt{S}_\mu \mathtt{p}, \, \mathtt{p}^T  \mathtt{q} =0, \, \forall \mathtt{q}\in \mathtt{ker(S_\mu)} \quad  \Leftrightarrow \quad \mathtt{\hat{U}_\mu^*(K_\mu +\tfrac{1}{2}I)} \mathtt{\hat{V}_\mu}\mathtt{P} = \sigma \hat{\Sigma}_\mu\mathtt{P}.
\end{equation} 
 The $(N-\ell) \times (N-\ell) $ generalized eigenvalue problem for $\mathtt{P}$ is well-conditioned, since $\hat{\Sigma}_\mu$ is invertible.
However, the enumeration of Steklov-Helmholtz eigenvalues for the original problem requires that we append $\sigma_1=\sigma_2=...=\sigma_{\ell} = -\infty$. These correspond to the (infinite magnitude) eigenvalues associated with the Steklov-Helmholtz problem at an exceptional wavenumber $\mu_D$ of multiplicity $\ell$.

In practice, we use the {\it truncated} SVD instead, dropping singular values smaller than a specified tolerance TOL (and corresponding singular vectors). To improve efficiency, we first compute the SVD for the eigenvalue problem on a coarse grid with $N_1=N/3$ for $N>300$; this allows us to identify rank-deficiency. Of course, it is possible that we {\it miss} rank-deficiency of the single layer matrix if the grid with $N_1$ points fails to resolve the corresponding eigendensity. 

\begin{algorithm}
	\caption{BIO-MOD: a wavenumber-robust method for Steklov-Helmholtz}\label{alg:SHspec}
	\begin{algorithmic}
		\Require Curve parametrization, wavenumber $\mu$, $N$, TOL.
		\If{ $N>300$} \State $N_1=N/3$ \Else \State $N_1=N$ \EndIf
		\State Build discrete single layer matrix $\mathtt{S_\mu}$ using $N_1$ discretization points
		\State Compute SVD $\mathtt{S_\mu = U_\mu\Sigma V_\mu, \quad \Sigma = diag(s_1,\cdots,s_{N_1})}$
		\State Compute $\ell = \#\{\lfloor \log_{10}(\mathtt{s_i})\rfloor \leq \log_{10}(TOL)\}$.
		\If{$\ell\neq 0$ } 
	\State Build matrices $\mathtt{S_\mu}, (\mathtt{K_\mu +\frac{1}{2}I})$ using $N$ discretization points
	\State Compute SVD $\mathtt{[U_\mu,\Sigma_\mu,V_\mu ]= S_\mu}$  
		\State Compute truncated SVD $\mathtt{S_\mu  = \hat{U}_\mu \hat{\Sigma}_\mu \hat{V}_\mu}$ by discarding singular values $<TOL$
		\State Solve generalized EVP $\mathtt{\hat{U}_\mu^*(K_\mu +\frac{1}{2}I)} \mathtt{\hat{V}_\mu}\mathtt{P} = \sigma \mathtt{\hat{\Sigma}_\mu}\mathtt{P}$
		\State Return eigenpairs $\sigma, {\tt p}=\mathtt{\hat{V}_\mu P}$. Append $\sigma_1=\sigma_2...\sigma_\ell = -\texttt{Inf}$
	\Else
	\State Build matrix $(\mathtt{K_\mu +\frac{1}{2}I})$ using $N$ discretization points
     \If{$N_1<N$}	\State  Build matrices $\mathtt{S_\mu}$ \EndIf
	\State Solve generalized EVP $\mathtt{(K_\mu +\frac{1}{2}I)}\mathtt{p} = \sigma \mathtt{S}_\mu \mathtt{p}$
	\State Return eigenpairs $\sigma, {\tt p}$
		\EndIf
	
\end{algorithmic}
\end{algorithm}

We note that for wavenumbers which are {\it not} exceptional, the BIO and BIO-mod algorithms reduce to the same method. 

We close this section by noting that Algorithm BIO-MOD will work for Steklov-Helmholtz problems in 3D as well; the boundary integral formulation and quadratures involved must, of course, be changed suitably.

\subsection{Numerical convergence tests}
In this section we demonstrate numerically the convergence properties of the proposed BIO-MOD method. The rates of convergence will depend on the accuracy of the quadrature methods. There are two important considerations while assessing performance: 
\begin{itemize}
	\item How well are the first $k$ eigenvalues approximated, as the number of discretization points $N$ is increased? For Lipschitz domains, what role does the degree of polynomial grading play?
	In the experiments below, we report the (mean) relative error of the first Q Steklov-Helmholtz eigenvalues computed using $N$ discretization points,defined as
	\begin{equation}\label{eq:MRE}
	\text{MRE}_{Q}(N):=\frac{1}{Q} \sum_{j=1}^{Q} \frac{|\sigma_{i,N}(\mu) - \sigma_i(\mu)|}{{|}\sigma_i(\mu){|}}
	\end{equation} (If $\sigma_i(\mu)=0$, we use the absolute error instead.)
	\item As $N$ increases, we are able to compute an increasing number of approximate eigenvalues. The higher discrete eigenvalues may not be good approximations to the corresponding true ones.  How many eigenvalues are well-approximated for a given choice of $N$? 
\end{itemize}

Away from exceptional wavenumbers $\mu$, the boundary integral approach yields high-accuracy approximations of the Steklov eigenpairs without modification. For domains with smooth boundary the convergence is spectral.  In \autoref{fig:diskconvergencebasic} we show $\text{MRE}_{16}(N)$ on the unit disk as well as a (non-convex) kite-shaped domain. 
For the unit disk the true (unsorted) Steklov-Helmholtz eigenvalues are given by 
$\mu\frac{J_n'(\mu)}{J_n(\mu)}$. {For} the kite-shaped domain whose boundary has parametrization \begin{equation}\label{eq:kite}
 \text{Kite}: \, \Gamma :=\{(x,y)\vert  x=\cos(t)+\kappa\cos(2t)-\kappa, y=1.5\sin(t), \, t\in [0,2\pi) \}, 
\end{equation} 
with $\kappa=0.65$, the computed eigenvalues are compared with a highly-resolved computation ($N=2048$).

Regularity of the density plays a key role in recovering eigenvalues to desired precision. For smooth densities, very few boundary points are required. From \autoref{fig:diskconvergencebasic}, \autoref{tab:100sig_p1}, and \autoref{tab:100sig_30} we also see that high accuracy can be achieved with a few boundary points, even for large wavenumbers, or higher into the Steklov-Helmholtz spectrum, as soon as the density is resolved. In \autoref{tab:100sig_p1} we present computed approximations for $\sigma_{100}(0.1)$ on the unit disk, and in  \autoref{tab:100sig_30} we present approximations for $\sigma_{100}(30)$. We see that even with $\mu=30$, 200 boundary points suffice to resolve the large eigenvalues correct to 14 digits. 
 This accuracy is not achievable with similar numbers of unknowns if we use discretization strategies which rely on resolving the eigenfunctions in the volume $\Omega$. 

\begin{figure}[H]
	\centering
	\includegraphics[width=0.495\linewidth]{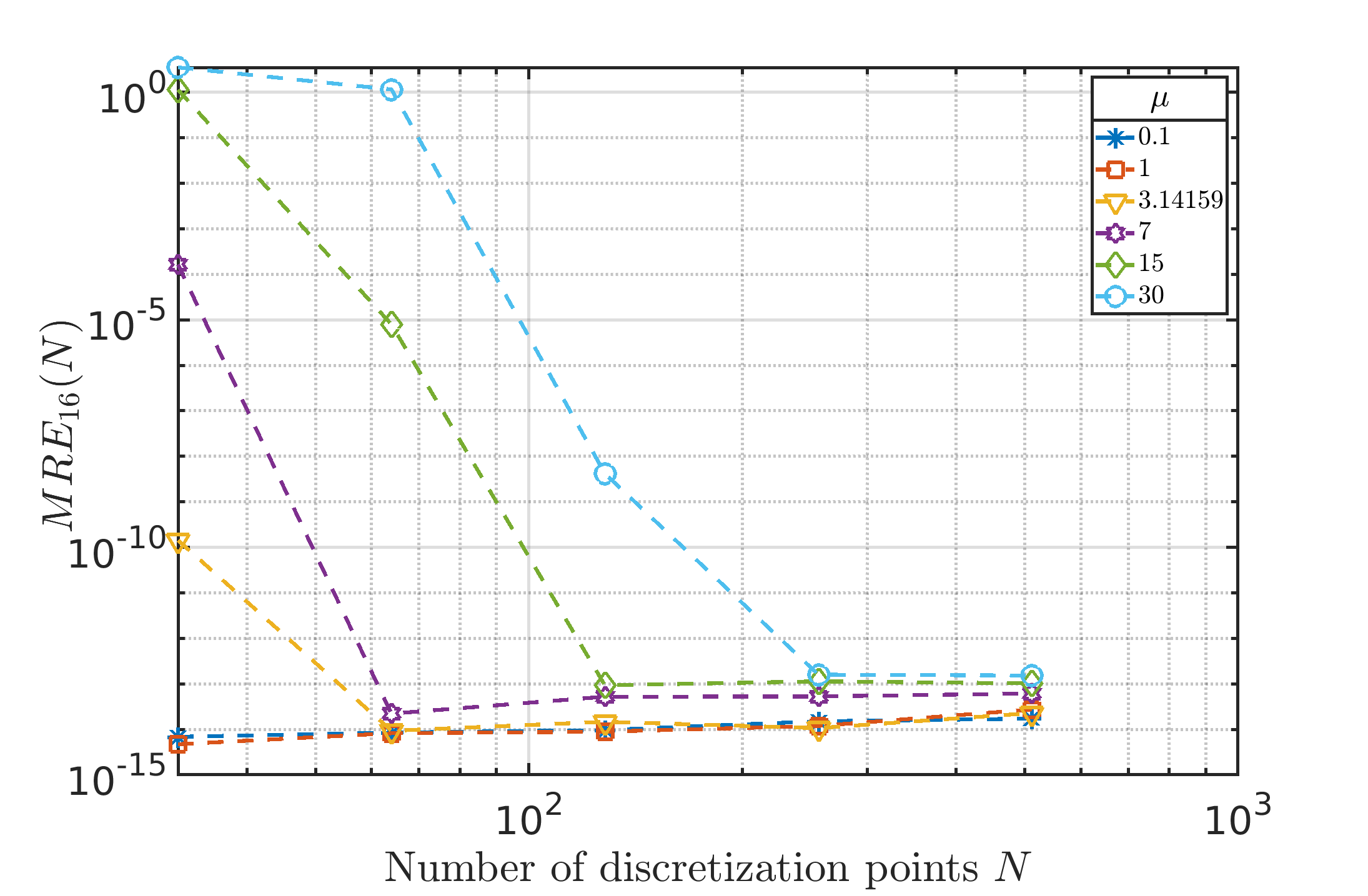}
		\includegraphics[width=0.495\linewidth]{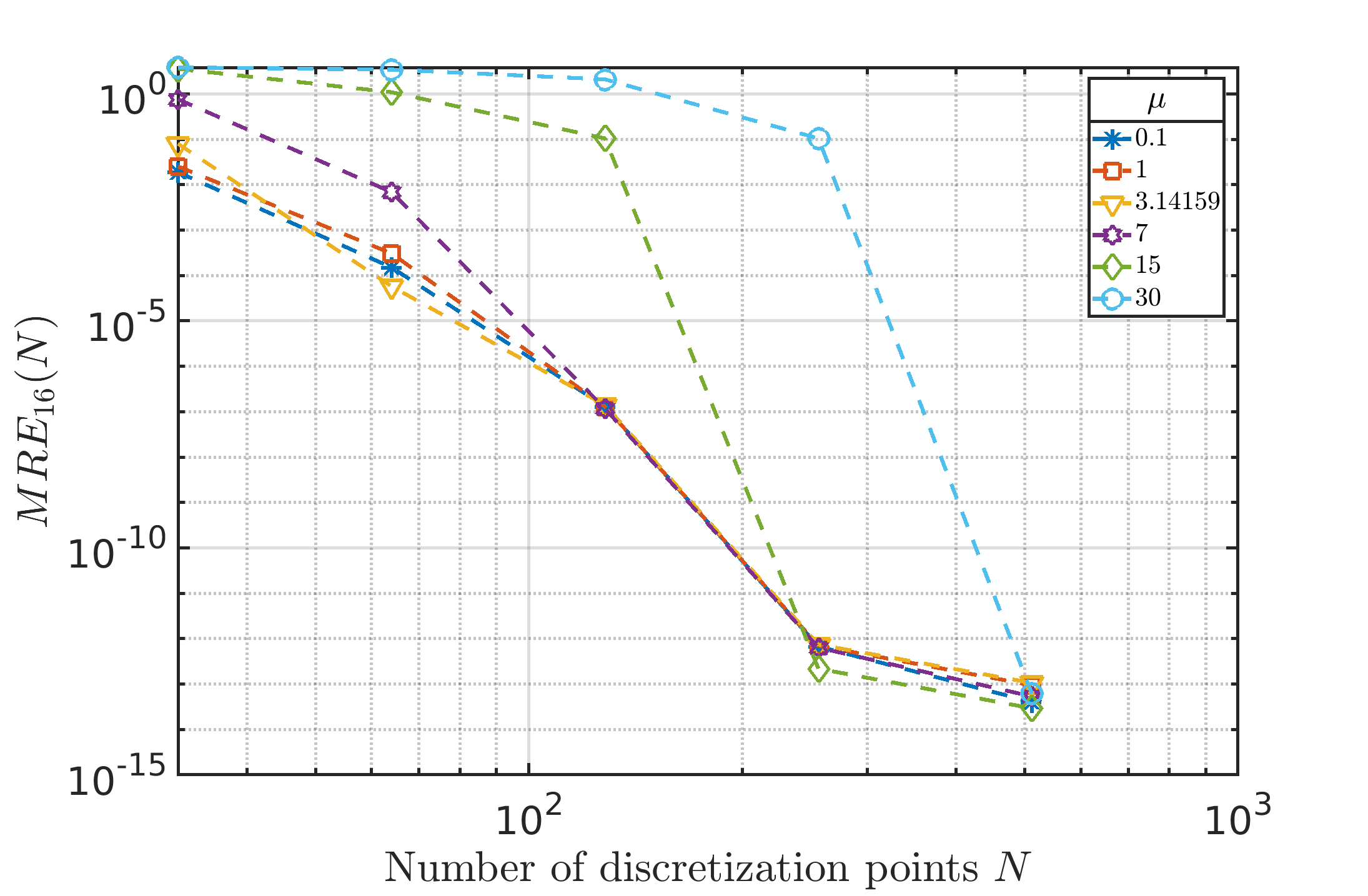}
	\caption{Convergence of first 16 Steklov-Helmholtz eigenvalues $\sigma_i(\mu)$ on (L) the unit disk and (R) a kite-shaped region \autoref{eq:kite}. Shown is the mean relative error $\text{MRE}_{16}(N)$ (see \autoref{eq:MRE}). Spectral convergence is observed. The number of points $N$ required to achieve a prescribed accuracy clearly depends on the wavenumber $\mu$. All $\mu$ here are non-exceptional.}
	\label{fig:diskconvergencebasic}
\end{figure}
\begin{table}[H]
	\parbox{.46\linewidth}{
		\centering
		\begin{tabular}{ |c||c|c|  }
			\hline 
			$\mu$ & $N$ & $\sigma_{N,100}(\mu)$\\
			\hline
			0.1 &  True &  {\bf{ 49.99990196069045}}	\\
			&  100 &   {\bf{49.99}}500649892375\\
			&  120 & {\bf{49.99990196069}}154	\\
			&  140 & {\bf{49.99990196069}}178\\
			\hline
		\end{tabular}
		\caption{The $100^{th}$ eigenvalue for $\mu = 0.1$ on the unit disk.}
		\label{tab:100sig_p1}
	}
	\hfill
	\parbox{.45\linewidth}{
		\centering
		\begin{tabular}{ |c||c|c|  }
			\hline 
			$\mu$ & $N$ & $\sigma_{N,100}(\mu)$\\
			\hline
			30 &  True &   {\bf{43.91970692071435}}	\\
			&  160 & {\bf{43}}.79333214528069	\\
			&  180 & {\bf{43.91970}}592453353	\\  
			&  200 & {\bf{43.919706920714}}10	\\
			\hline
		\end{tabular}
		\caption{The $100^{th}$ eigenvalue for $\mu = 30$.}
		\label{tab:100sig_30}
	}	
\end{table}
We next show numerical convergence of eigenvalues for a square of side length $\pi$ and two wavenumbers: $\mu = 7$ and $\mu=5$ (an exceptional wavenumber). {The true eigenvalues can be solved for by separation of variables. They are intersections of certain tangent and hyperbolic tangent functions dependent on $\mu$ and the side length, which we numerically approximate to high precision in \texttt{Maple}}. We used both graded mesh quadrature and the RCIP quadrature (implemented in {\tt ChunkIE}). For the graded mesh quadratures, we use grading degrees  $p=3,4,5,6$. The convergence results are presented in \autoref{fig:square} and  \autoref{fig:biomodsquarecomparsion} respectively. We note from \autoref{fig:square} that some  eigenvalues are approximated correctly almost to machine precision even high into the spectrum; these correspond to smooth eigendensities. At the (exceptional) wavenumber $\mu=5$, recall that  the single layer matrix is not invertible. Nonetheless, using Matlab's {\tt eig} we present relative errors for the first 90 eigenvalues. In the left subplot we document the performance of a graded mesh approach and the RCIP quadratures. On the right, we combine the BIO-MOD approach with these two different quadrature approaches. We can see that the better conditioning of the single and double layer matrices in the RCIP approach leads to better accuracy. Recall that in the BIO-MOD approach we discard {\it all} the singular values of the single layer matrix below a threshold; if the matrix is poorly conditioned (as in the graded mesh case), BIO-MOD will remove important information which then impacts accuracy. 
\begin{figure}
\centering
		\includegraphics[width=0.495\textwidth]{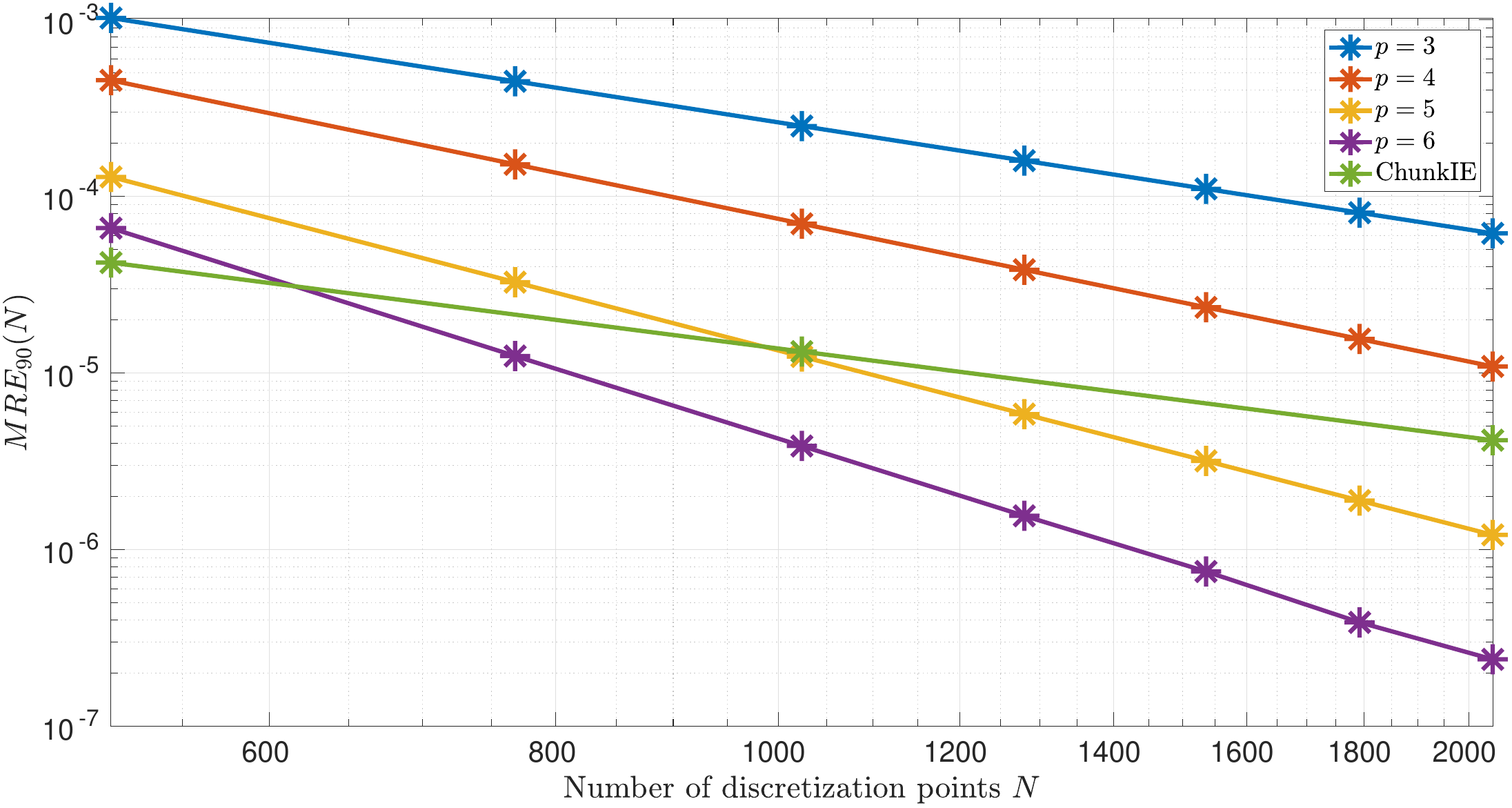}              
		\includegraphics[width=0.495\textwidth]{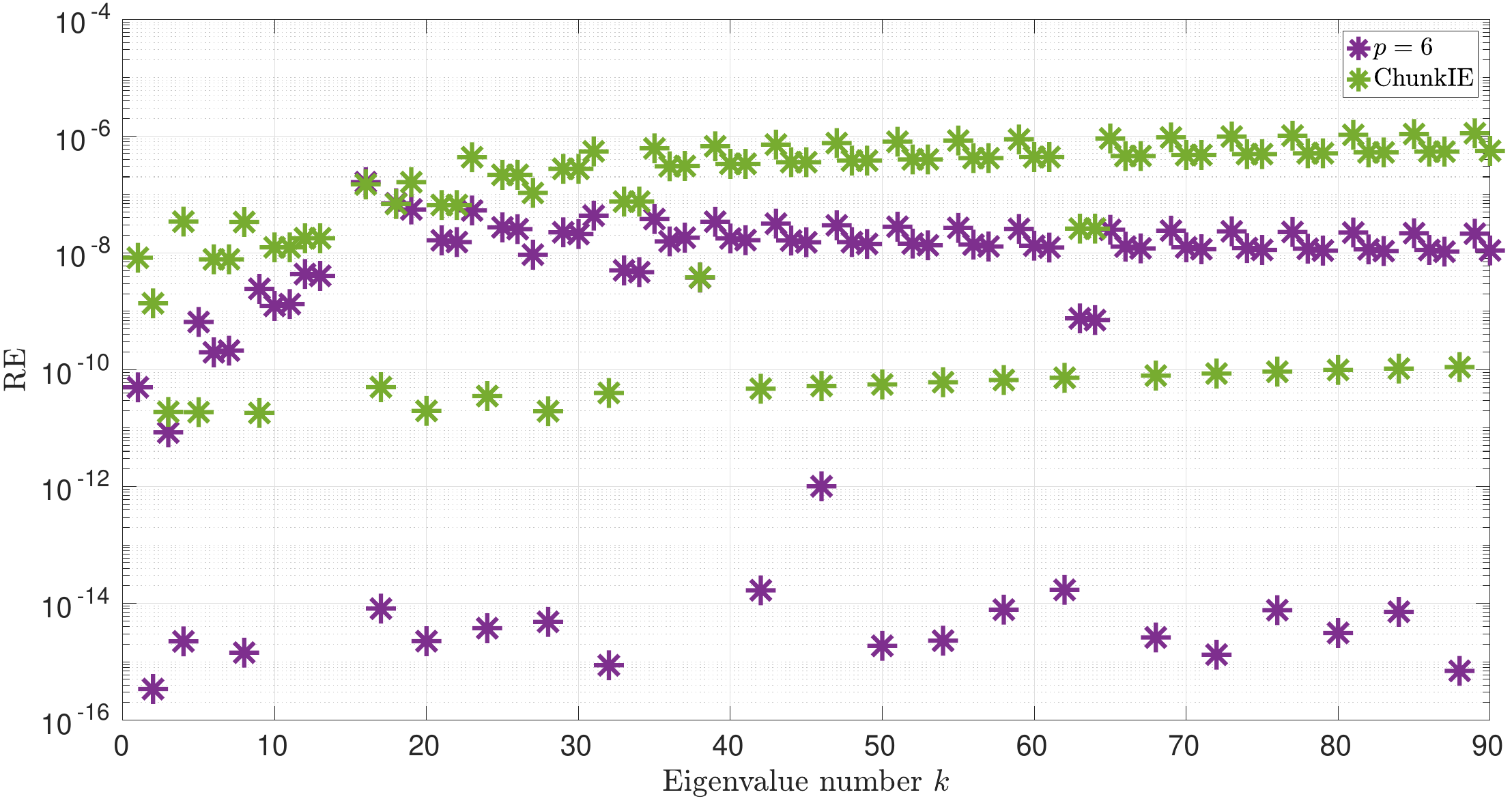}              
	\caption{Relative errors for Steklov-Helmholtz eigenvalues on a square of side length $\pi$, non-exceptional wavenumber $\mu=7$. (L) Plots of $\text{MRE}_{90}(N)$ v/s $N$ are shown for polynomially-graded meshes of degrees 3-6, as well as the RCIP approach (see \autoref{eq:MRE}). (R) The  relative errors of the first 90 eigenvalues is shown for a graded mesh approach with $p = 6,$ $N=2048$ points, and $TOL=1e-14$ in the BIO-MOD method.}   
	\label{fig:square}
\end{figure}
\begin{figure}\centering
		\includegraphics[width=0.495\linewidth]{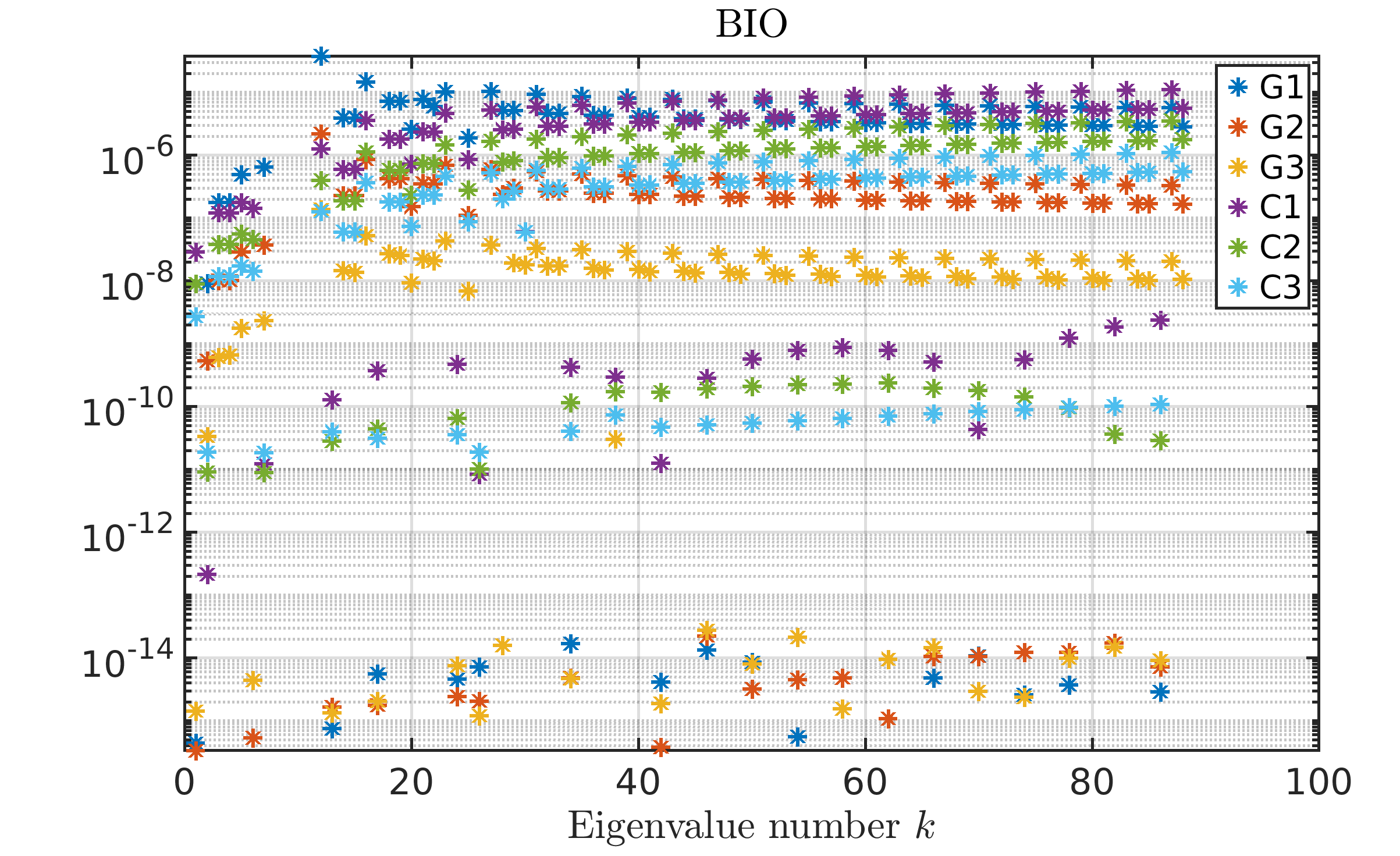}
	\includegraphics[width=0.495\linewidth]{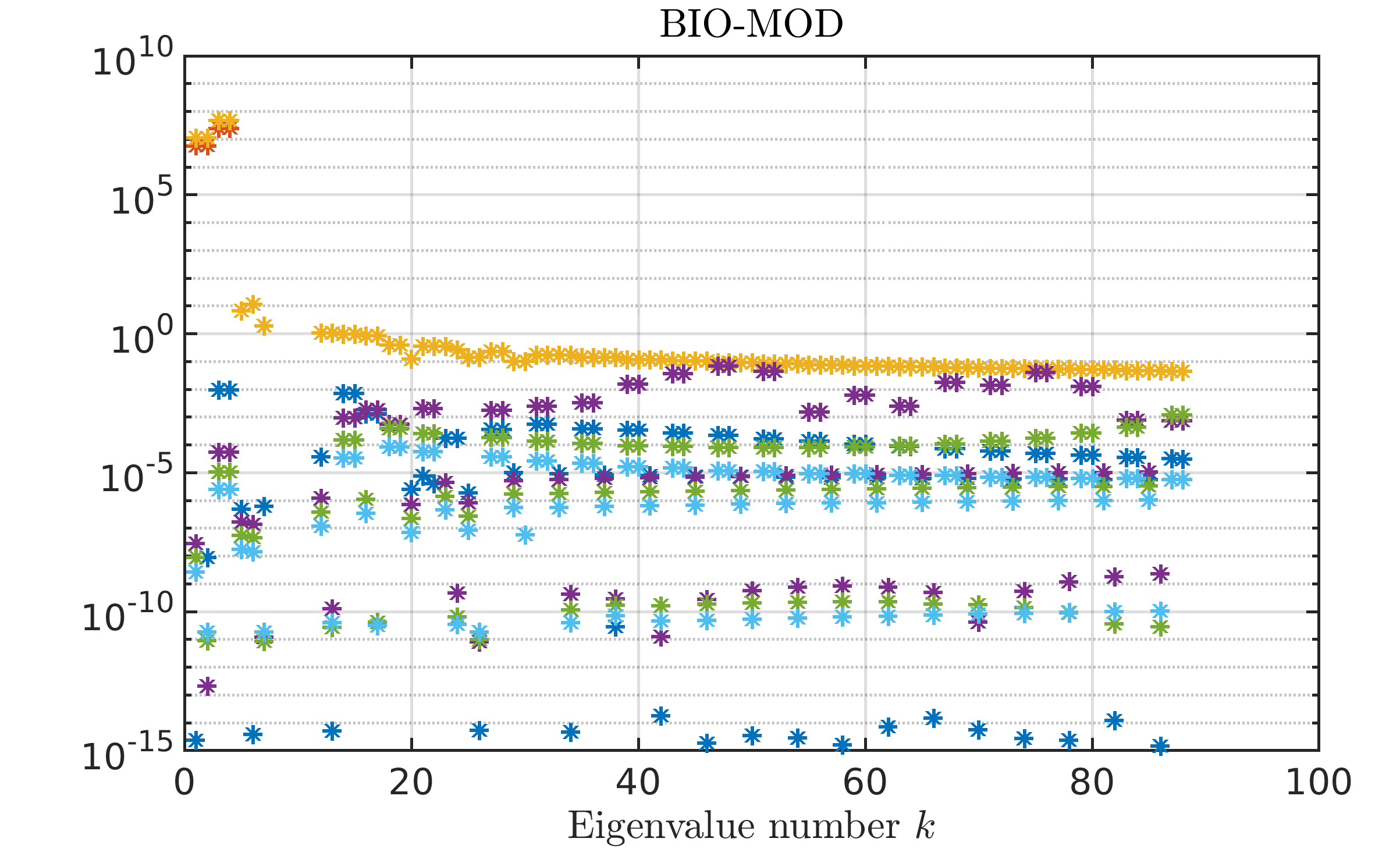}
	\caption{Relative errors of the first 88 (finite) eigenvalues on a square, wavenumber $\mu=5$ is exceptional.  We document the performance of the RCIP-based (`C') and the graded mesh quadratures (`G'), using $N=512 (1), 1024 (2), 2048 (3)$ quadrature points in both methods. Left: using the regular BIO approach (and removing the large negative eigenvalues resulting from the exceptional wave number). Right: BIO-MOD approach.}
	\label{fig:biomodsquarecomparsion}
\end{figure}
We now demonstrate robustness of the BIO-MOD approach even for wavenumbers close to exceptional ones. We first observed that for the unit disk, both the BIO and the BIO-MOD approaches give comparably accurate results, even very close to the exceptional wavenumbers ({naturally this is $TOL$ dependent}). So the BIO-MOD approach confers no advantage in this setting. However, the situation changes for smooth domains with less symmetry.
In \autoref{fig:alg1_other}, we first located an interior Dirichlet eigenvalue $\mu_D^2$ for two domains $\Omega_A$ (from \cite{barnett2014fast}), $\Omega_B$, whose boundaries have the polar parametrizations 
\begin{align*}
\Omega_A: \, \Gamma &:= \{r(\theta) = 1+0.3\cos(3(\theta+0.2\cos \theta)),\,\, \theta \in [0,2\pi)\},\\
\Omega_B: \, \Gamma &:=\{r(\theta) =(\exp(\cos \theta)\cos^2(2\theta)+\exp(\sin \theta)\sin^2(2\theta)), \theta \in [0,2\pi) \}.
\end{align*}

We then compute the Steklov-Helmholtz eigenvalues $\sigma_k(\mu_D)$ using both the BIO and the wavenumber-robust BIO-MOD approach. We might expect only part of the spectrum (the large negative eigenvalues) to not be correctly computed. However, the rank-deficiency of $\mathtt{S_\mu}$ polluted {\it all} the computed eigenvalues.
\begin{figure}[H]
	\begin{subfigure}[b]{0.45\textwidth}
		\centering
		\includegraphics[width=\textwidth]{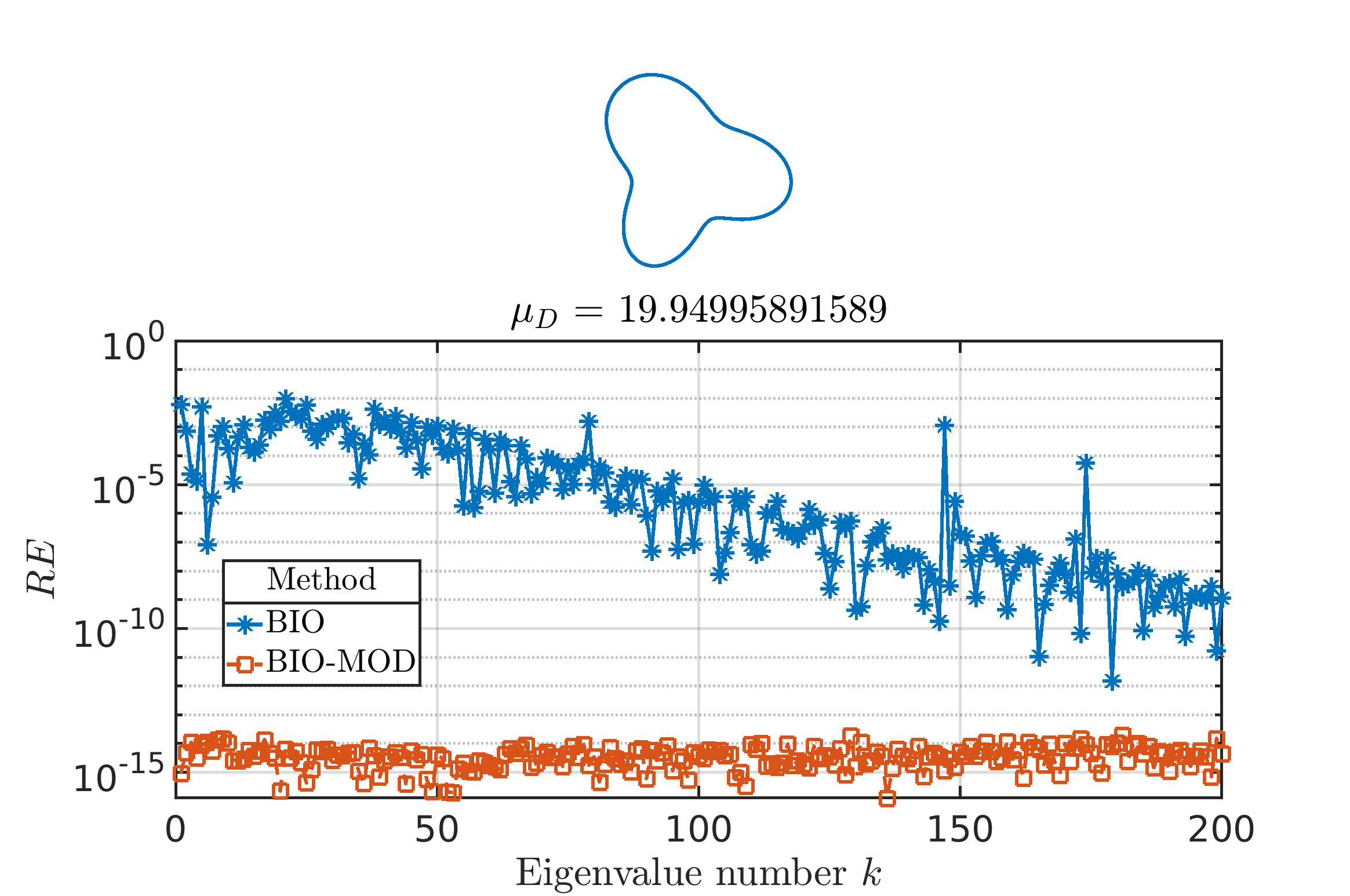}   
		\label{fig:alg1_asyp}
	\end{subfigure}
	\hfill
	\begin{subfigure}[b]{0.45\textwidth}
		\centering
		\includegraphics[width=\textwidth]{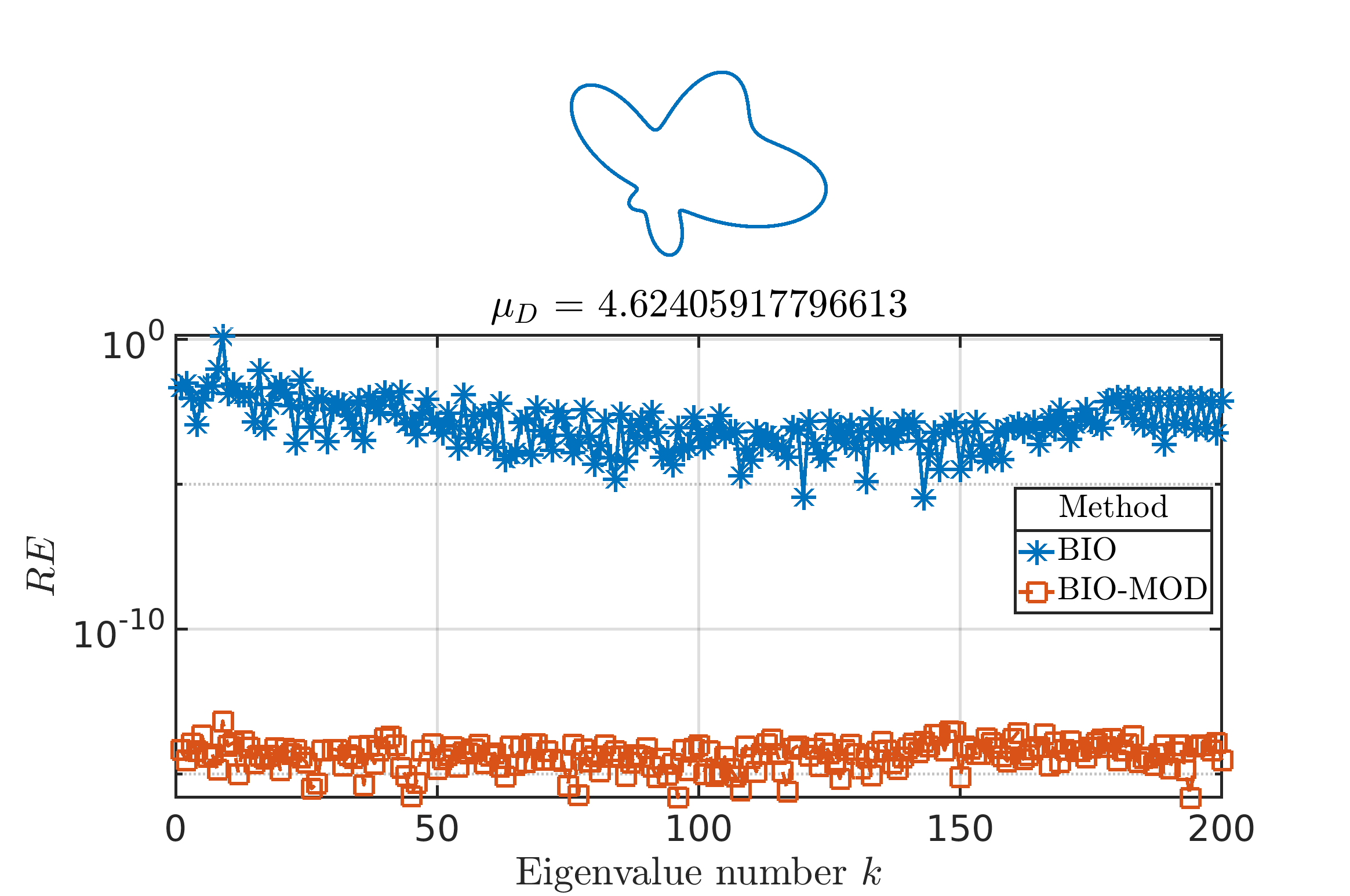}  
		\label{fig:alg1_but}
	\end{subfigure}
	\caption[] { Relative errors of the first 200 Steklov-Helmholtz eigenvalues $\sigma_{k}(\mu_D)$ at exceptional wavenumbers $\mu_D$ for (left) $\Omega_A$ and (right) $\Omega_B$. Tolerance is set at $1e-10$ for the BIO-MOD algorithm, and the number of discretization points $N=1024$ for both. Eigenvalues are computed using the standard BIO approach (blue) and the wavenumber robust approach BIO-MOD (red). Computed eigenvalues are compared to highly refined calculations using BIO-MOD. }   
	\label{fig:alg1_other} 
\end{figure}
The use of boundary integral operators easily allows us to consider eigenvalue problems with {different boundary conditions on different boundary pieces $\Gamma_L$. In particular, we test the sloshing eigenvalue problem where we impose}  Neumann conditions on part of the boundary, and Steklov conditions on the rest. High accuracy is retained. In the next example, we consider a semicircle (radius 1) with $\frac{\partial u}{\partial n}=0$  on the straight side and the Steklov boundary condition on the curved side. The wavenumber is set to $\mu=8$. The spectrum  $\sigma_k(8)$ is a subset of the Steklov-Helmholtz spectrum of the disk (corresponding to eigenfunctions with reflectional symmetry across the straight edge). In \autoref{fig:sc2}, we show the relative errors for the computed eigenvalues. We observe uniform recovery of 10 digits for $p = 6$ and $N \geq 600$.

\begin{figure}[H]
	\centering
	\begin{subfigure}[b]{0.49\textwidth}
		\centering
		\includegraphics[width=\textwidth]{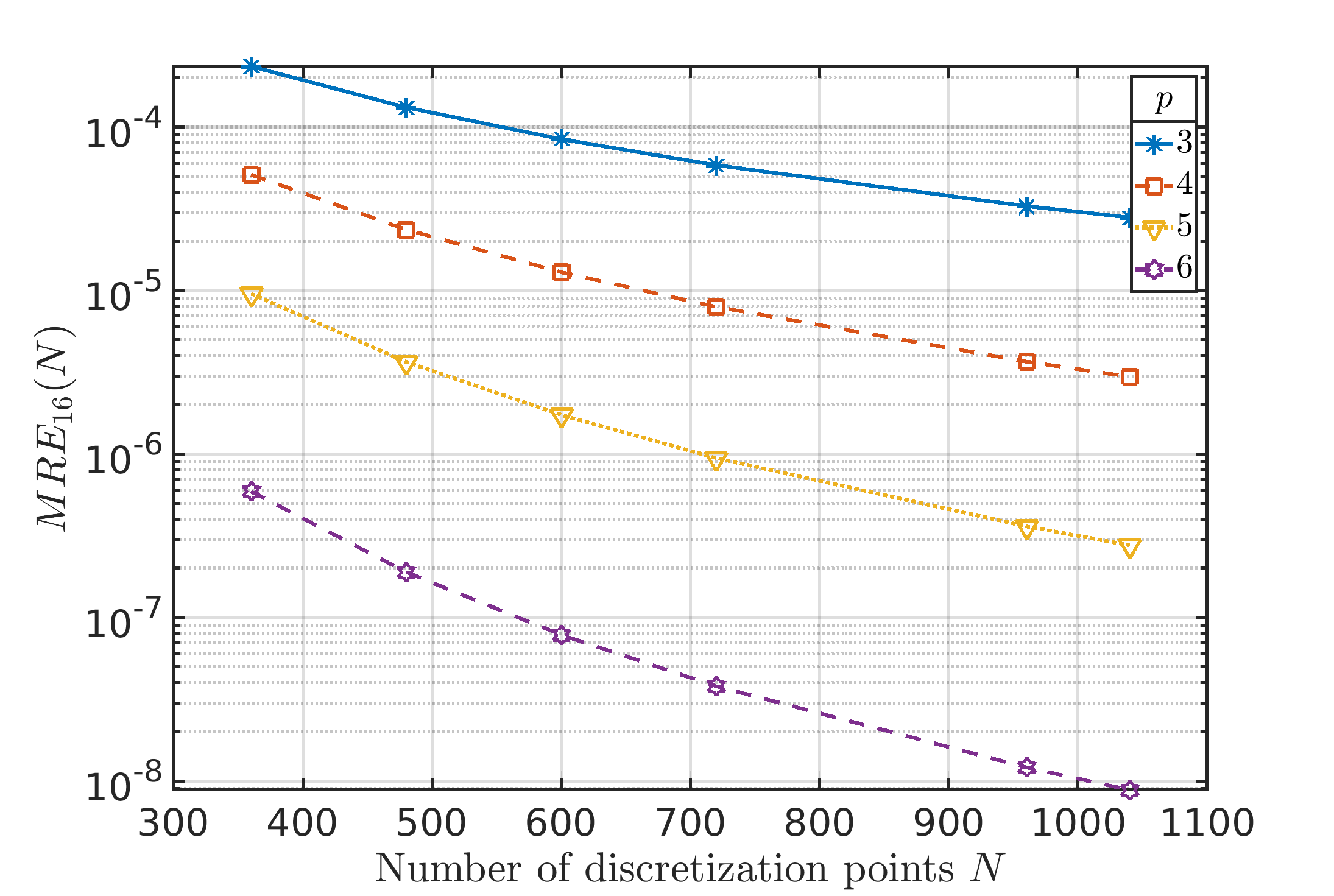}              \caption{{\small }}
		\label{fig:scn}
	\end{subfigure}
	\hfill
	\begin{subfigure}[b]{0.49\textwidth}
		\centering
		\includegraphics[width=\textwidth]{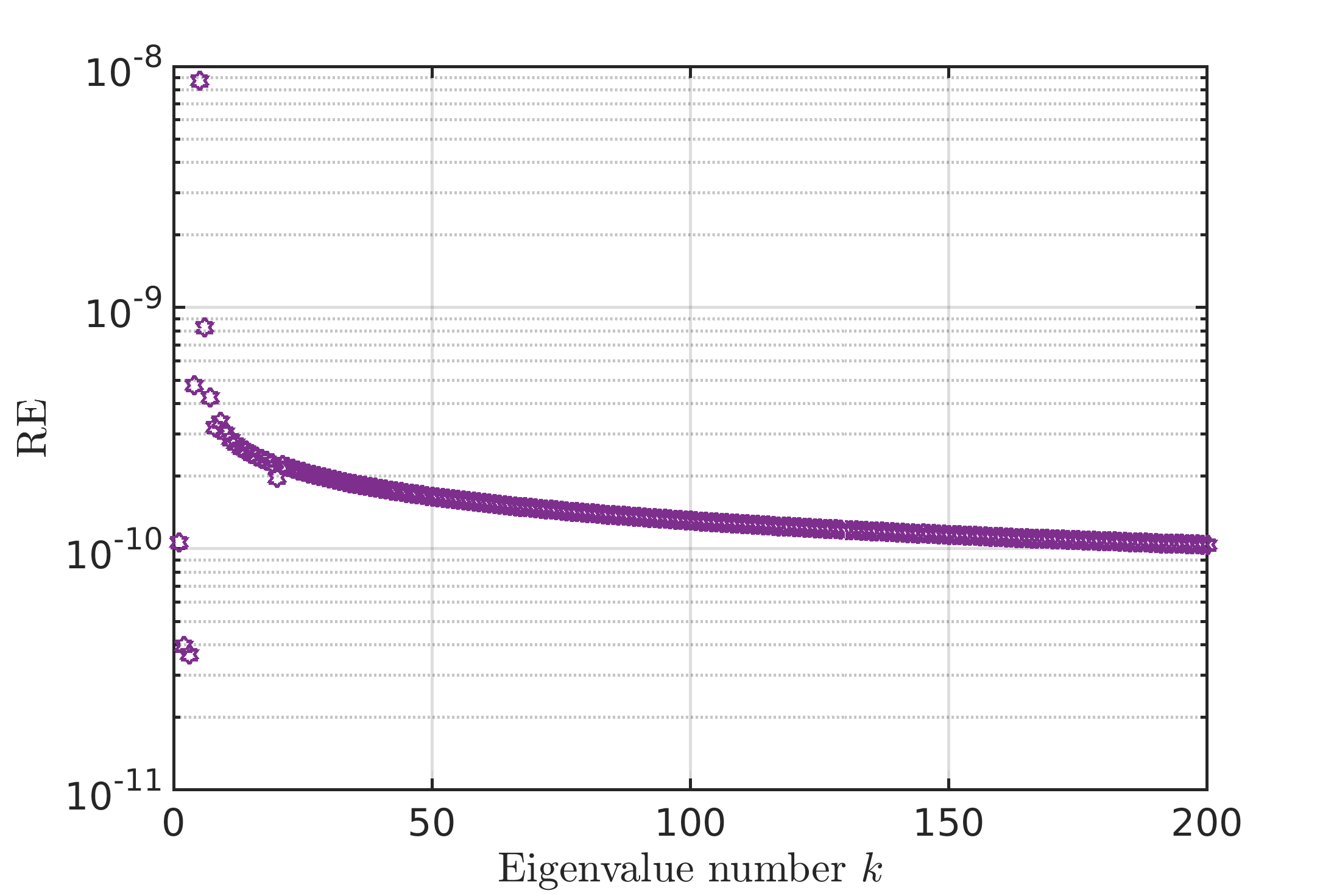}              
\caption{}
		\label{fig:scre}
	\end{subfigure}
	\caption[]                        {{\small Steklov-Neumann eigenvalues on a semicircle with radius 1, and wavenumber $\mu=8$.  (a) Norm of relative error of the first Steklov-Neumann 16 eigenvalues as a function of the total number of boundary points $N$. Different curves correspond to different choices of polynomial grading degree. (b) Relative errors for the $kth$ Steklov-Neumann eigenvalues, $k=1,\dots,200$.}}   
	\label{fig:sc2}
\end{figure}

Boundary integral approaches are successful even for challenging problems involving re-entrant corners, where eigenfunctions may have poor regularity. In \autoref{tab:l_ns} we present $\sigma_k(\mu), k=1,\dots,6$ for the L-shaped domain in \cite{Ma22}, with {wavenumber $\mu=2$}. We compare the computed eigenvalues with those reported in Table 3 in \cite{SunTurner} and Table 6 of \cite{Ma22}. In \cite{SunTurner}, the authors use Lagrange finite elements and an extension of the spectral indicator approach, and convergence is expected to be quadratic in the mesh size for regular eigenfunctions. An  integral operator approach via a direct ansatz is used in combination with the spectral indicator method in \cite{Ma22}, though it appears a Nystr\"om discretization was used for this domain. This may explain the two eigenvalues which appear missing. In \autoref{tab:l_ns} we fix $p=6$ and vary $N$. In both, the {`F}ine {G}rid' {(FG)} eigenvalues correspond to $N=1200,~p=6$.

\begin{table}
{{
	\centering
	\begin{tabular}{ |c||c|c|c|c|c|c|  }
		\hline 
		$N$         & $\sigma_1$ & $\sigma_2$ & $\sigma_3$ & $\sigma_4$ & $\sigma_5$ & $\sigma_6$ \\
		\hline
		FG & \bf{-2.5332135} & \bf{-0.8577893}  & \bf{-0.1245247}  & \bf{1.0852970}  & \bf{1.0911950}  & \bf{1.4169010}  \\
		320       & {\bf{\textcolor{blue}{-2.533}}}1625 & {\bf{\textcolor{blue}{-0.8577}}}702  & {\bf{\textcolor{blue}{-0.124}}}6061  & {\bf{\textcolor{blue}{1.08}}}48137  & {\bf{\textcolor{blue}{1.09}}}04050  & {\bf{\textcolor{blue}{1.416}}}4033  \\
		640       & {\bf{\textcolor{blue}{-2.5332}}}110 & {\bf{\textcolor{blue}{-0.85778}}}83  & {\bf{\textcolor{blue}{-0.12452}}}87  & {\bf{\textcolor{blue}{1.0852}}}737  & {\bf{\textcolor{blue}{1.0911}}}569  & {\bf{\textcolor{blue}{1.416}}}8767  \\
		960       & {\bf{\textcolor{blue}{-2.533213}}}2 & {\bf{\textcolor{blue}{-0.857789}}}1  & {\bf{\textcolor{blue}{-0.12452}}}52  & {\bf{\textcolor{blue}{1.08529}}}42  & {\bf{\textcolor{blue}{1.09119}}}05  & {\bf{\textcolor{blue}{1.416}}}8981  \\
		\hline
	\cite{SunTurner}	& -2.533099  & -0.857457  & -0.124494  & 1.085374   & 1.091319   & 1.417098   \\
		\hline
	\cite{Ma22}	& -2.5332    & -0.8578    & -0.1246    & -          & 1.0909     & -          \\
		\hline
	\end{tabular}
	\caption{Convergence of first 6 eigenvalues $\sigma_k(\mu)$, $k = 1,\dots,6$ of the L-shaped domain with $\mu^2 = 4$. Here, degree of polynomial grading $p = 6$ is fixed. `Fine Grid’ (FG) eigenvalues correspond
		to $N = 1200$. Final 2 rows include Steklov-Helmholtz eigenvalues reported in \cite{SunTurner} and \cite{Ma22}.}
	\label{tab:l_ns}
}
}	
\end{table}

\section{Some questions in spectral geometry}
We now use the BIO-MOD approach developed above for some applications in spectral geometry. We begin by examining questions of spectral asympotics. Our approach yields highly accurate and wavenumber robust approximations even for larger eigenvalues, and we can study how (say) the curvature of the boundary impacts the spectrum. We next examine some questions in optimization, for which we first propose some scale-invariant spectral quantities to examine. Since the integral operator approach is very flexible, we are also able to examine a variant of Weinstock's conjecture on annular domains.

\subsection{Spectral asymptotics}

The asymptotic behaviour of the Steklov-Laplace spectrum for both smooth domains and Lipschitz domains has been described extensively. We can use the BIO-MOD approach to compute the Steklov-Helmholtz spectrum to study the impact of wavenumber $\mu$ on the spectral asympotics. 

In \autoref{fig:goal1}, we show the first several Steklov-Helmholtz eigenvalues on a disk, an ellipse and a kite-shaped domain (all equi-perimeter, parametrized as in \eqref{eq:kite}).
\begin{figure}
\centering
\includegraphics[width=.32\textwidth]{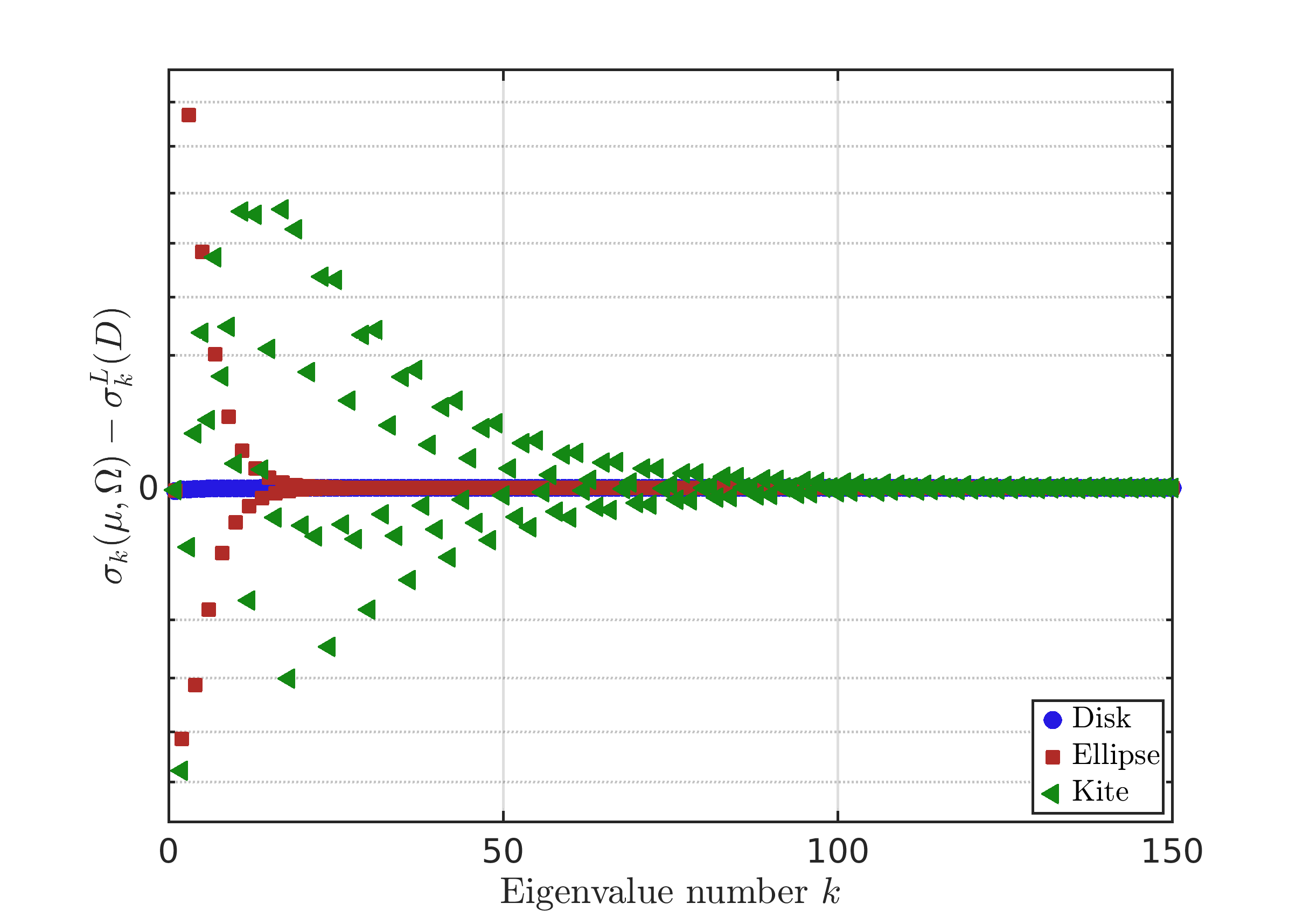}
	\includegraphics[width=.32\textwidth]{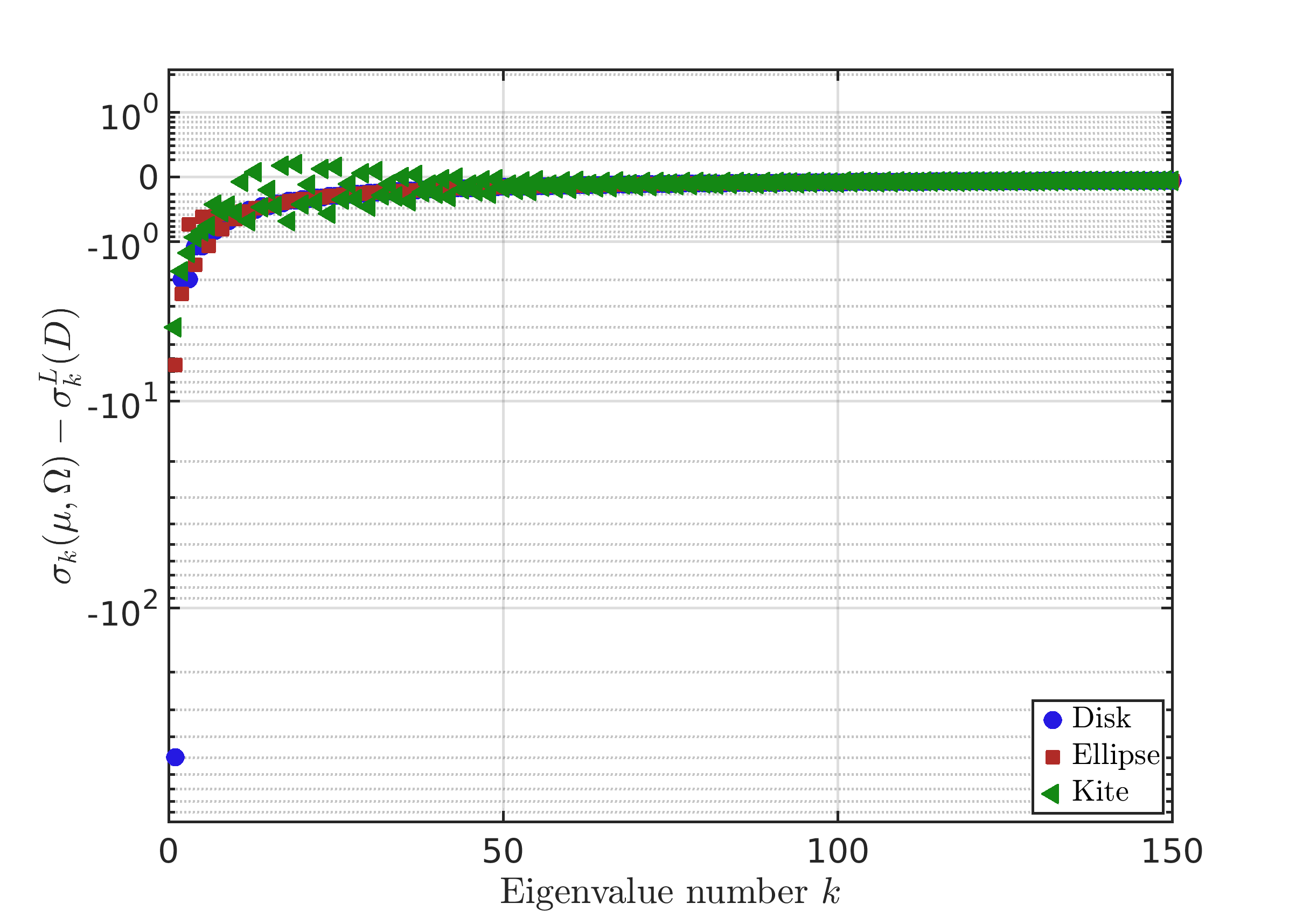}
	\includegraphics[width=.32\textwidth]{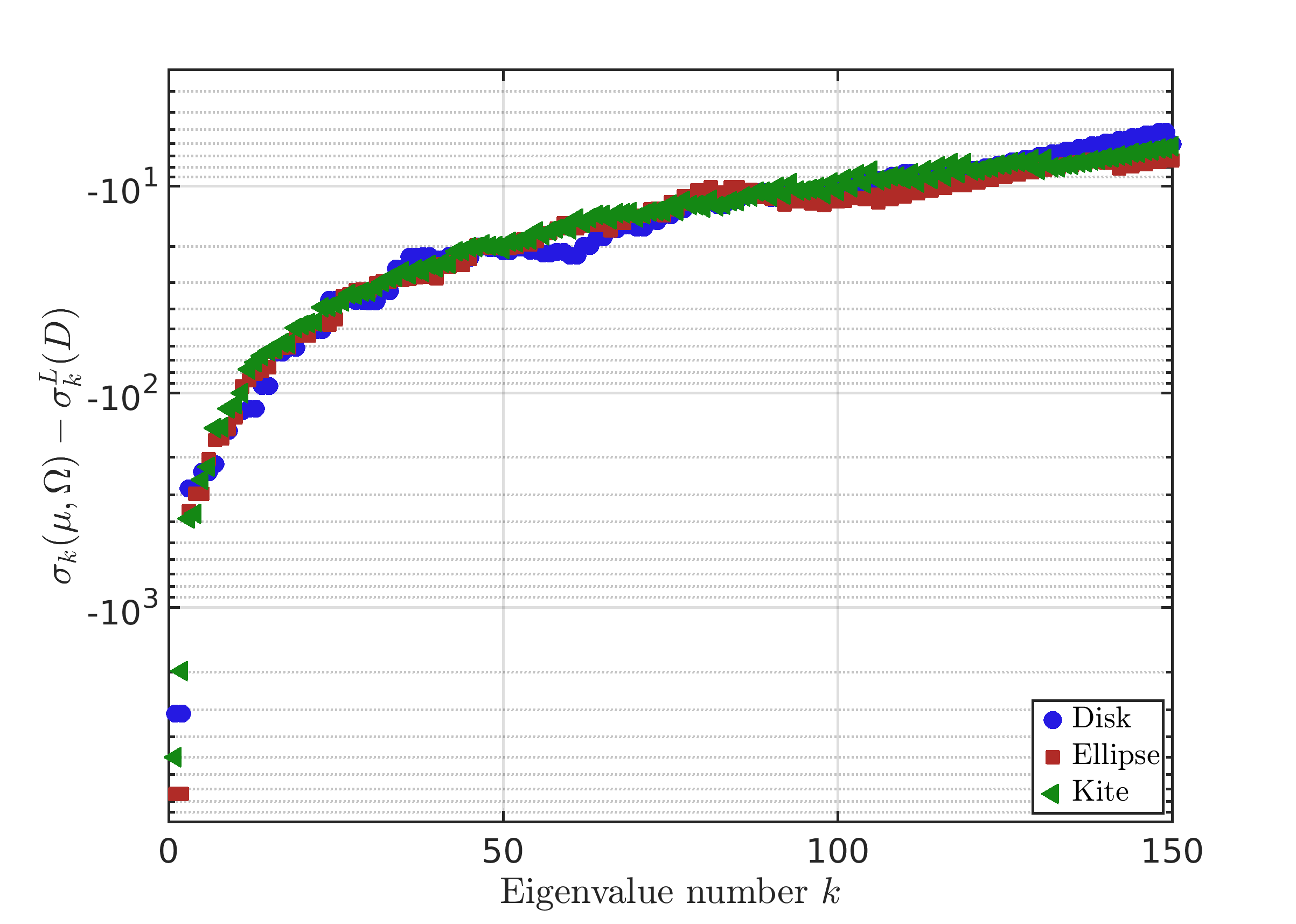}
	\caption{Spectral asymptotics on smooth domains. We plot $\sigma_k(\mu,\Omega) - \sigma^L(\mathbb{D})$ between the Steklov-Helmholtz eigenvalue on $\Omega$ and the Steklov-Laplace eigenvalues on $\mathbb{D}$. Domains $\Omega$ are $\mathbb{D}$ (blue), and equiperimeter ellipse (red) and an equiperimeter kite \eqref{eq:kite}. Wavenumbers  (Left) $\mu=0.1$, (Middle) $\mu=2.4$, (Right) $\mu = 43$. Recall all the Steklov-Laplace eigenvalues $\sigma^L_k\geq 0$. } 
	\label{fig:goal1}
\end{figure}
In \autoref{fig:goal1} we present 3 plots comparing $\sigma_k(\mu)-\sigma_k^L(\mathbb{D}), k=1,\dots,150$ for three choices of wavenumbers. In each subplot, we record $\sigma_k(\mu,\Omega)-
\sigma^L(\mathbb{D})$ for $\Omega${;}  equiperimeter disk, ellipse and `kite'.
From the left subplot of \autoref{fig:goal1}, we see that the Steklov-Helmholtz spectra very quickly become close to that of the Steklov-Laplace problem. This is not surprising, since the wavenumber $\mu=0.1$. We point out that the deviation of the `kite' domain spectrum persists into the spectrum. 
In the middle subplot \autoref{fig:goal1}, the chosen wavenumber $\mu=2.4$ is close to an exceptional value for the disk; $\sigma_1(2.4, \mathbb{D}) - \sigma_1^L(\mathbb{D}) = \sigma_1(2.4, \mathbb{D})$ is a large negative number. For large $\mu$ (right subplot) the effects of domain curvature become clearly visible in the pre-asymptotic regime.

We next examine the impact of boundary curvature on the Steklov-Helmholtz spectrum for a kite-shaped domain with parametrization as in \eqref{eq:kite}, for $\mu$ non-exceptional. We vary the parameter $\kappa$, and record the perimeter-scaled eigenvalues  $|\Gamma|\sigma_k(\mu), k=1,\dots,20$
in \autoref{fig:goal2a}; beyond this value, the variation in the spectra between domains is very small.

\begin{figure}
	\centering
	\includegraphics[width=0.5\linewidth]{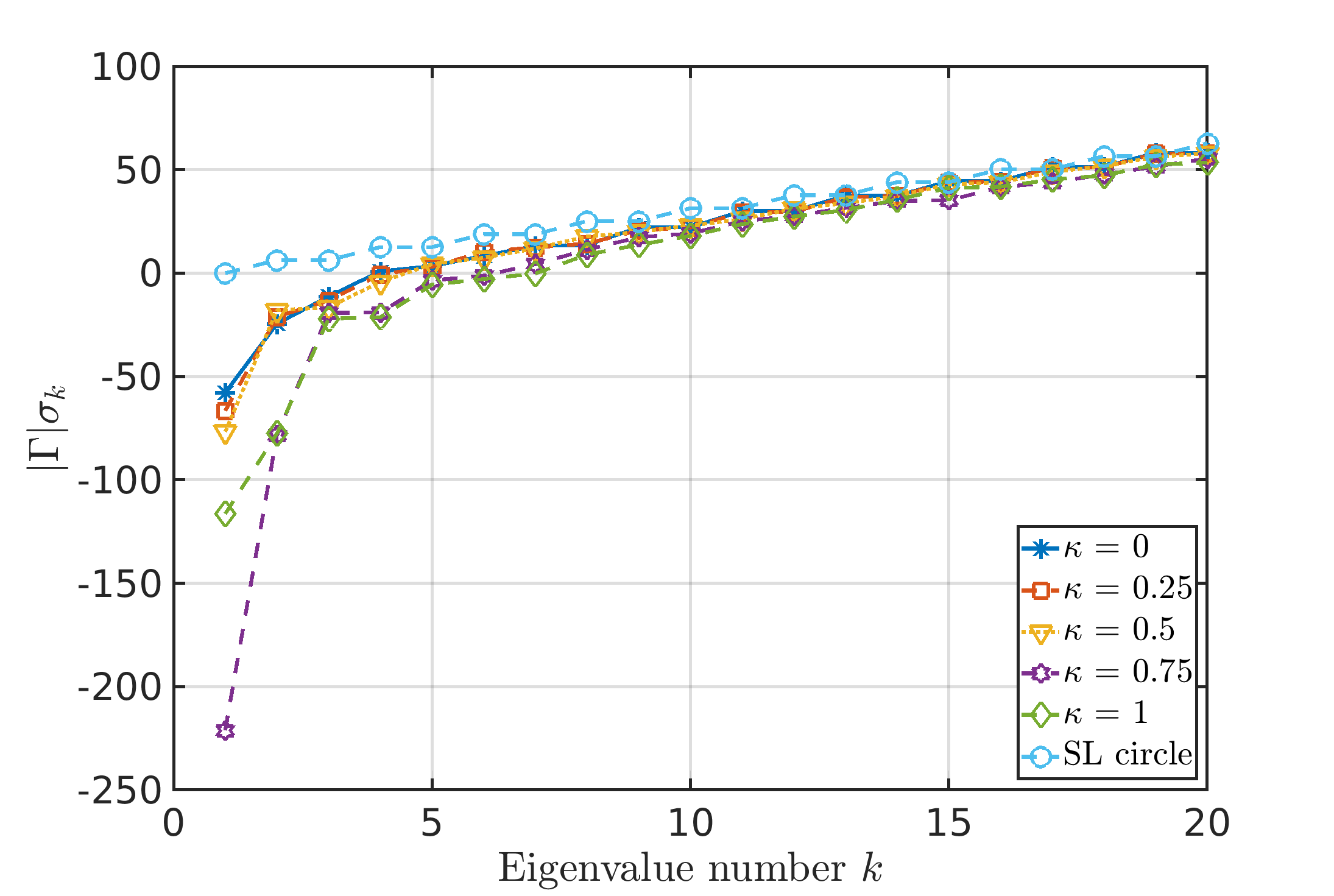}
	\caption{Impact of boundary curvature on spectral asymptotics. In this figure, we show $|\Gamma_\kappa|\sigma_k(\mu, \Omega_\kappa)$ for a range of kite-shaped domains. In this experiment the wavenumber $\mu=\pi$ and $N=500$.}
	\label{fig:goal2a}
\end{figure}
The spectral asymptotics of the Steklov-Laplace eigenvalues on curvilinear polygons has been precisely described in \cite{SherQM}, with remarkably precise approximations via {\it quasimodes} that are described in terms of the sidelengths and angles of polygons. The authors provide explicit formulae for quasimodes $QM_k(\Omega)$ for some shapes $\Omega$, and it is shown that
$$ \sigma_k(0,\Omega) - QM_k (\Omega) \rightarrow 0.$$
We are unaware of similar results for the Steklov-Helmholtz problem. 
 
In \autoref{fig:goal3mu5} we record the variations
 $$ \frac{|\sigma_k(\mu,\Omega) - QM_k(\Omega)|}{QM_k(\Omega)}.$$ For a triangle with sides 1,1,$\sqrt{2}$, there are two sequences of quasimodes: $QM_k= \pi k$ approximating multiplicity 2 eigenvalues, and $QM_k=\pi/\sqrt{2} (k-1/2)$, $k\in \mathbb{N}$. On a square with sides 1, the quasimodes approximate multiplicity-4 eigenvalues, and are of form
$$	 QM_k=(k-1/2)\pi.$$ As expected, we see the quasimodes eventually (in $k$)  approximate the true spectrum. The pre-asymptotic regime very clearly depends on the wavenumber.
\begin{figure}
	\centering
		\includegraphics[width=0.45\textwidth]{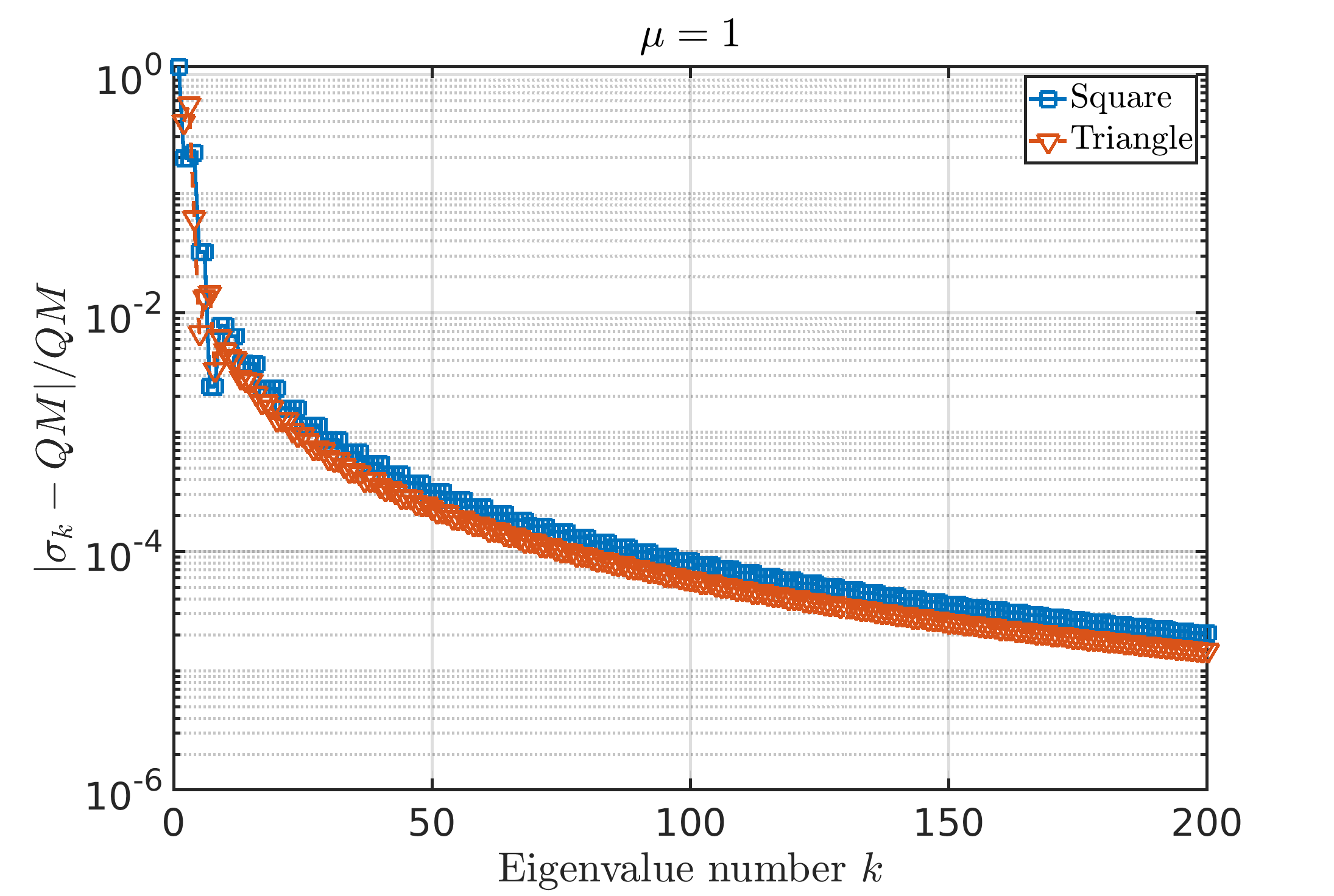}
	\includegraphics[width=0.45\textwidth]{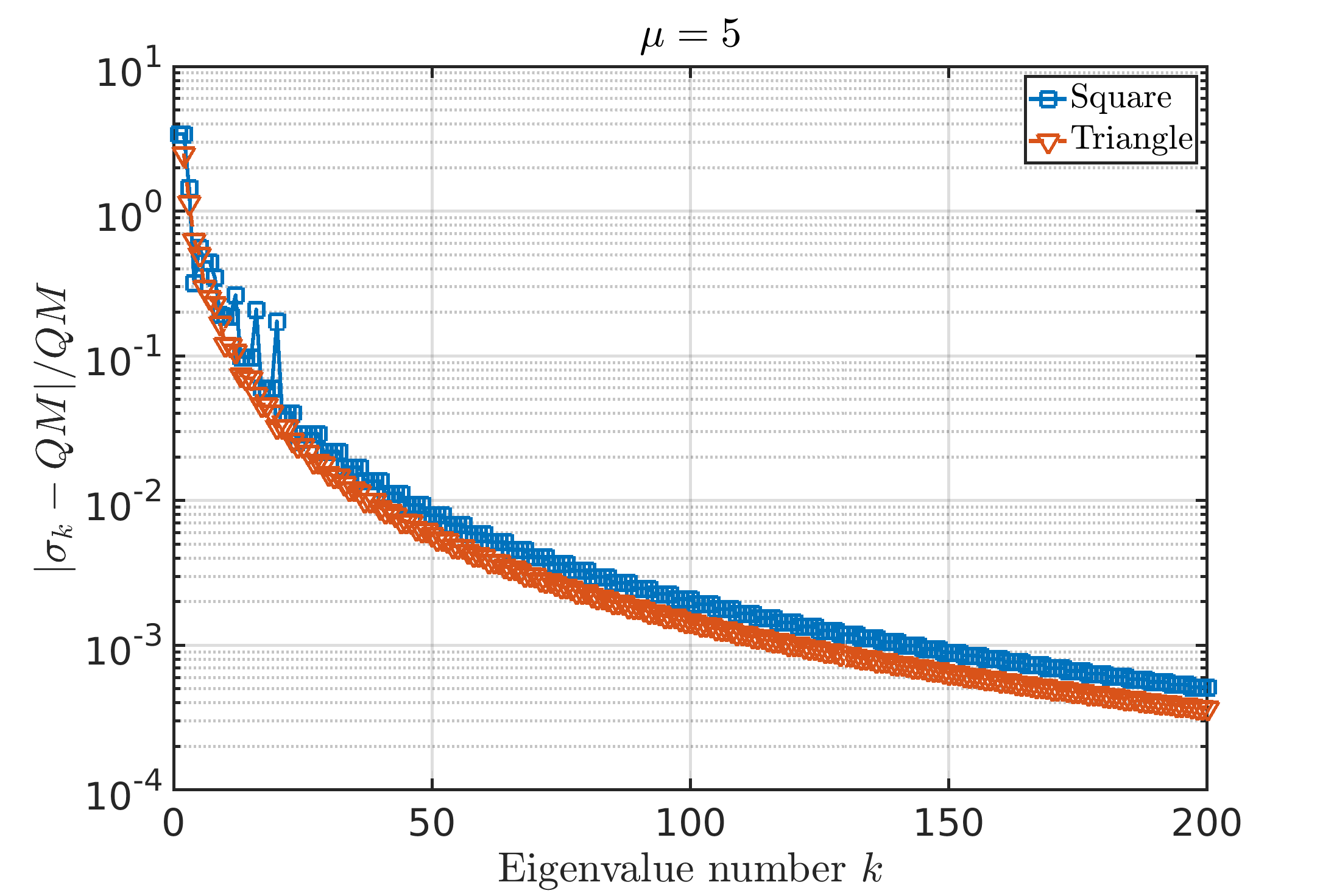}	
	\caption{Do the Steklov-Laplace quasimodes $QM_k(\Omega)$ capture the asymptotics of $\sigma_k(\mu,\Omega)?$ We report $| \sigma_k(\mu,\Omega) -QM_k(\Omega)|/QM_k(\Omega)$ v/s eigenvalue number $k$, for $ \Omega=$ a square and an isoceles triangle.  Left: $\mu=1$. Right: $\mu=5$. }
	\label{fig:goal3mu5}
\end{figure}

The Steklov-Helmholtz spectrum includes a (finite number of) negative eigenvalues, which depend on the wave number as well as the domain shape. In \autoref{fig:goal4}, we plot the number of negative eigenvalues as a function of wavenumber and perimeter for a range of shapes. We observe that as $\mu$ and $|\Gamma|$ increase, the number of negative values get closer to $\tfrac{|\Gamma|\mu}{2\pi}$.
\begin{figure}
	\centering
	\includegraphics[width=0.65\linewidth]{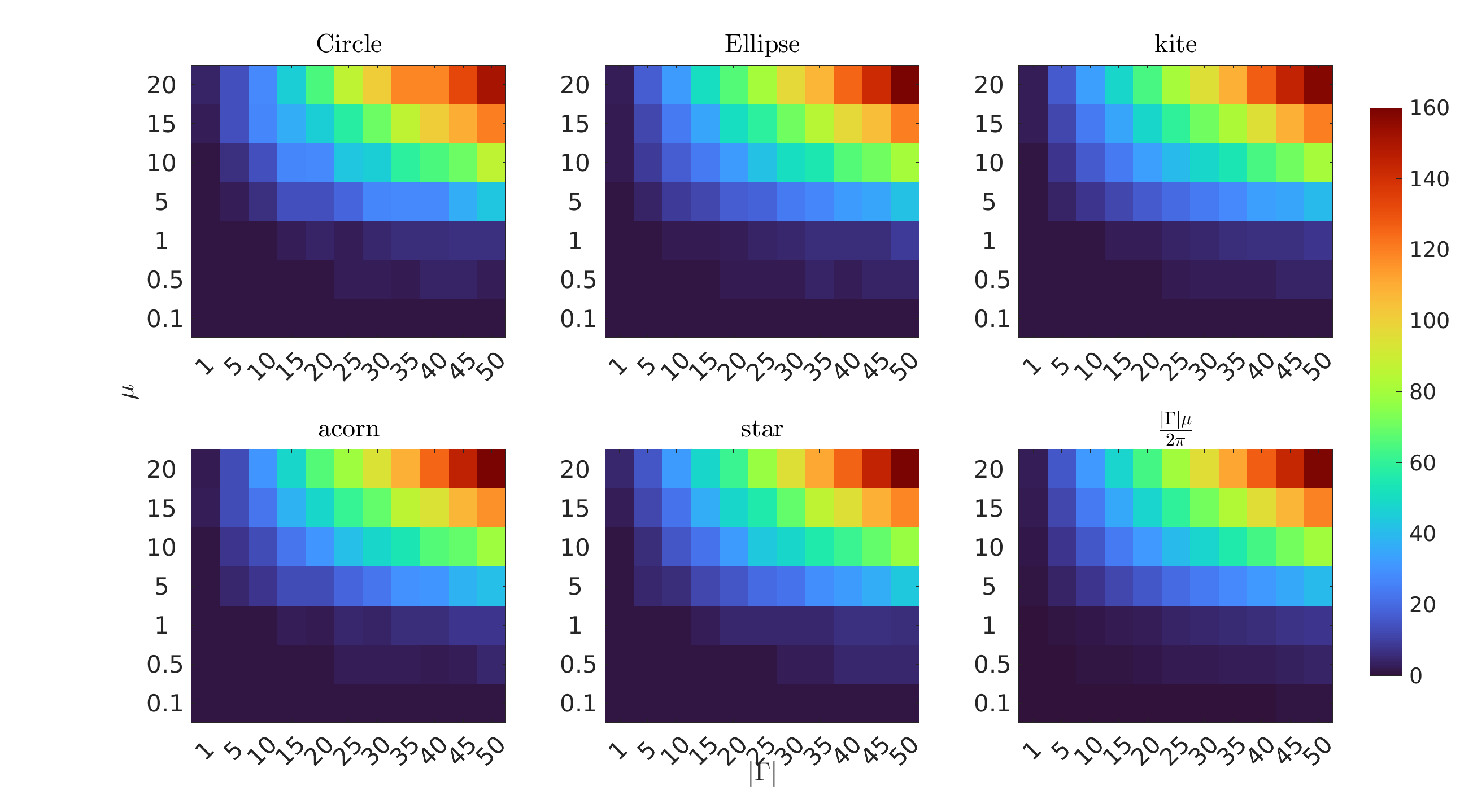}
	\caption{The number of negative eigenvalues in the Steklov-Helmholtz spectrum $\{\sigma_k(\mu,\Gamma)\}$ depends on wavenumber $\mu$ and perimeter. The shape of the domains does not seem to have as important. Bottom right figure is a plot of $\frac{|\Gamma|\mu}{2\pi}$.}
	\label{fig:goal4}
\end{figure}

\subsection{Spectral optimization}
A classical problem in spectral geometry concerns eigenvalue shape optimization: what is the optimizer of the $kth$ eigenvalue of an elliptic operator, under specific geometric constraints? Unsurprisingly, the answer crucially depends on the constraint. For instance, if we constrain the volume of domains $|\Omega|$ then the first Laplace-Dirichlet eigenvalue $\lambda_1$ is minimized by a ball. This fact is nicely expressed through the the Rayleigh-Faber-Krahn inequality, which in $\mathbb{R}^2$ reads:
\begin{equation}
	|\Omega|\lambda_1(\Omega)\geq \pi j_{0,1}^2
\end{equation} where $j_{0,1}$ is the first zero of the Bessel function of order zero; $j_{0,1}^2$ is the first Laplace-Dirichlet eigenvalue of the ball in 2D. 

 Under volume constraint, $\lambda_2$ is minimized by two identical balls. If we constrain instead the {\it perimeter} of $\Omega$, the minimizer of $\lambda_2$ is a regular convex domain whose boundary curvature vanishes at exactly two points.  The question of which constraints to impose for meaningful optimization problems for the Steklov-Helmholtz spectra is more delicate. For the (closely related) Robin spectral problems with eigenvalues $\mu_k$, a scale invariant quantity of interest is
$$ |\Omega|\mu_k(\sigma/|\Gamma|,\Omega );$$ other such quantities exist.  Observe that various scale-invariant quantities for the eigenvalues $\sigma(\Omega;\mu)$ of \eqref{eq:steklov-helmholtz} can be written in the form
\begin{equation}\label{eq:scale-invariance}
E_k(\mu,\Omega):= |\Omega|^{\alpha} |\Gamma|^\beta \sigma_k\left(\frac{\mu}{ |\Omega|^\gamma |\Gamma|^\delta},\Omega\right)
\end{equation}
where $(2\gamma+\delta)=1, (2\alpha+\beta)=1$. These constraints on $\alpha,\beta,\gamma$ and $\delta$ follow through dimensional arguments. 

During an optimization procedure over admissable domains, a given wavenumber $\mu$ may become an exceptional value, and then the quality of computed Steklov-Helmholtz eigenvalues degenerates. It is here the wavenumber-robust nature of our approach is particularly useful:  without the regularization approach we used, optimization routines tend to fail. In what follows, we discuss two specific scale invariant quantities where we restrict the wavenumber $\mu$ to a certain range with the help of the Faber-Krahn inequality. In this range there are no exceptional wavenumbers and in particular, the wavenumbers in this range are smaller than the first scaled Dirichlet-Laplace eigenvalue. From Theorem 7.4.2 in \cite{Levitin23} we then have that the admissible set is in $H^1(\Omega)$. Using a constant test function in the variational characterization of $\sigma_1$ described in {\sc Problem Ia'} we readily get
\begin{equation}\label{eq:bound1}
E_1\left(\mu,\Omega\right)\leq |\Omega|^{\alpha} |\Gamma|^\beta\frac{\int_\Omega \left[ - \frac{\mu^2}{|\Omega|^{2\gamma}|\Gamma|^{2\delta}}\right] }{\int_{\Gamma}1} = -\mu^2 |\Omega|^{\alpha-2\gamma+1}|\Gamma|^{\beta-2\delta-1}.
\end{equation}

\subsubsection{Perimeter constrained}
In the first set of experiments, we fix $\delta=1$, $\gamma=0$, $\alpha=0,\beta=1$. The objective function we consider is
\[
G_k(\mu,\Omega) := |\Gamma|\sigma_k\left(\frac{\mu}{|\Gamma|},\Omega\right).
\] We let $\Omega$ be an ellipse of eccentricity $e$. For a fixed wavenumber, therefore, $G_k(\mu,\Omega)$ is a function of $e$ for each $k$; these are plotted in the  top row of \autoref{fig:fig11b}. 
In these experiments, the (scaled) wavenumbers $\frac{\mu}{|\Gamma|}$ become exceptional if the $kth$ Dirichlet eigenvalue of $\lambda_k(\Omega) = \frac{\mu^2}{|\Gamma|^2}$. 

 From the Faber-Krahn inequality, we know that $|\Omega| \lambda_1(\Omega) \geq \pi j_{0,1}^2$. From the isoperimetric inequality, we know that $|\Gamma|^2  \geq {4\pi}|\Omega|.$ Putting these facts together, we see that if 
 $$ \frac{\mu_*^2}{|\Gamma|^2}\geq \lambda_1(\Omega) \geq \pi j^2_{0,1}\frac{1}{|\Omega|}\geq 4 \pi^2 j_{0,1}^2 \frac{1}{|\Gamma|^2},$$ then the range of wavenumbers $(0,\mu_*]$ includes at least one  exceptional wave number. Put differently, restricting $\mu \in  (0, 2\pi j_{0,1})  \approx (0, 15.104)$ will ensure we do not cross an exceptional wave number. From \autoref{eq:bound1} we then have $G_1(\mu,\Omega) \leq -\mu^2\tfrac{|\Omega|}{|\Gamma|^2}$.
 \begin{figure}
 	\centering 	\includegraphics[width=0.9\textwidth]{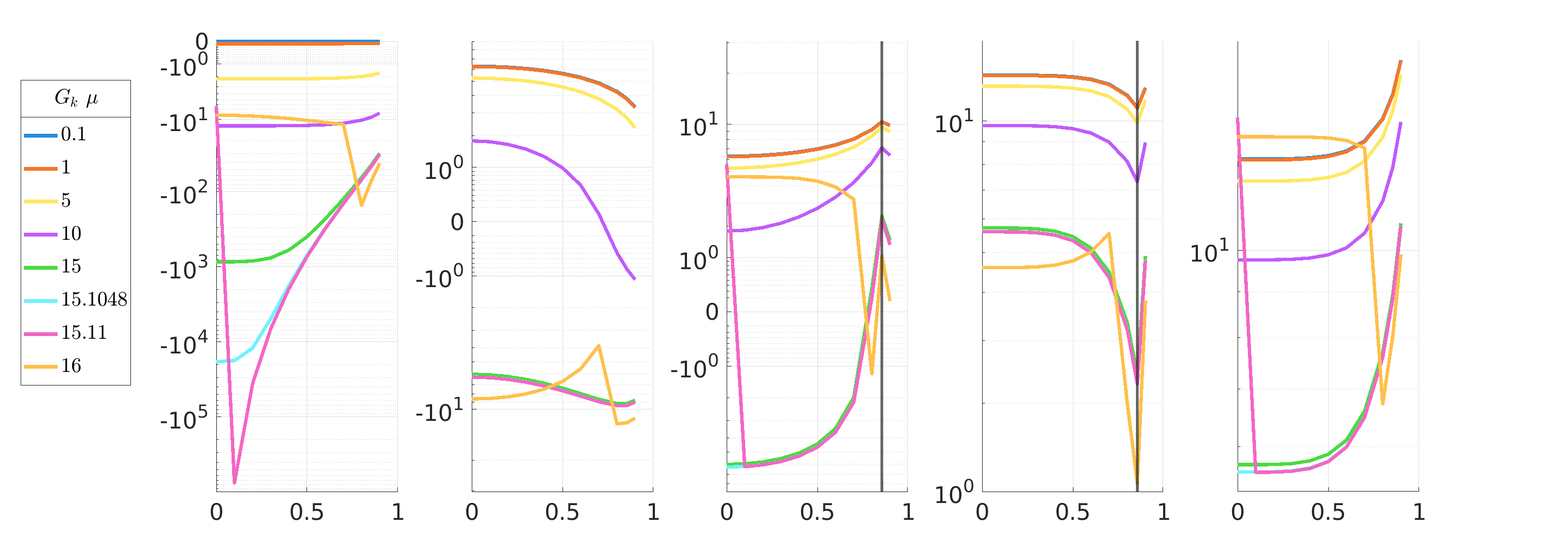} 	\includegraphics[width=0.9\textwidth]{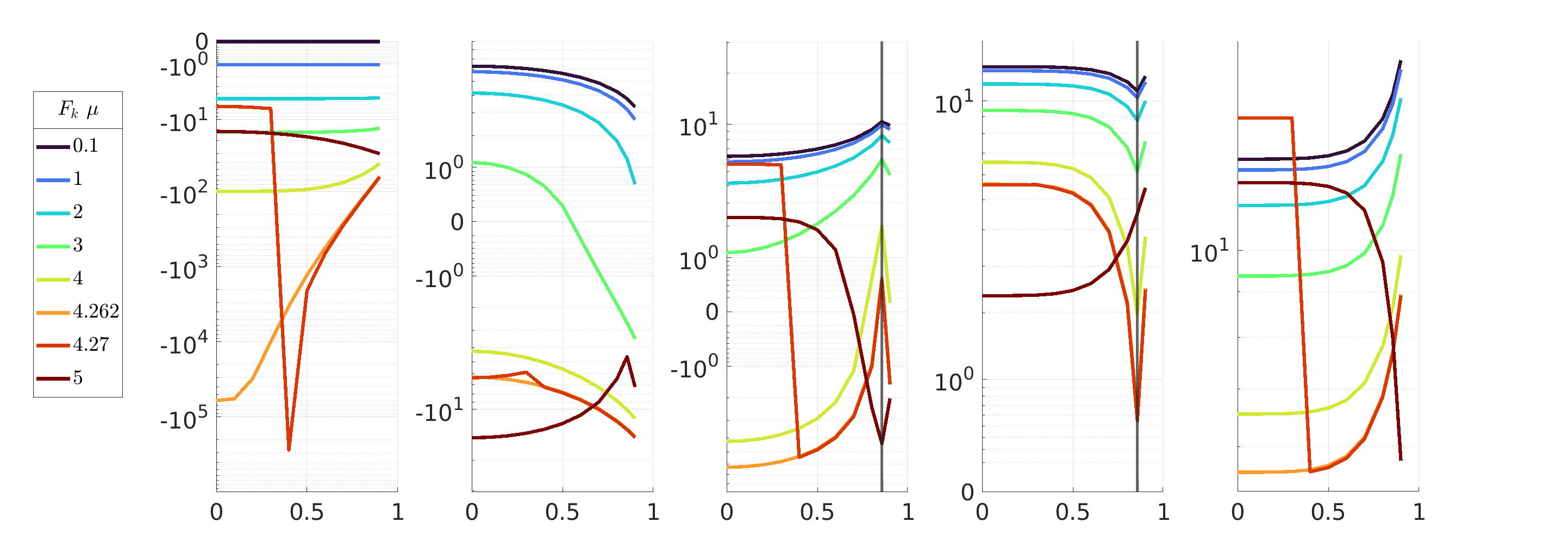} 	 	\caption{The spectral parameters (Top) $G_k(\mu,\Omega)$ and (Bottom) $F_k(\mu,\Omega)$ for $k = 1,\cdots,5$ on ellipses (parameterized by eccentricity $e$).}
 	\label{fig:fig11b}
 \end{figure}
\subsubsection{Volume constrained}
In the next set of experiments we report, we  fix $\delta=0$, $\gamma=1/2$, $\alpha=0,\beta=1$.  We optimize the spectral quantity
\[
F_k(\mu,\Omega) := |\Gamma|\sigma_k\left(\frac{\mu}{\sqrt{|\Omega|}},\Omega\right).
\]

Using the Faber-Krahn inequality, we restrict $\mu^2 \in (0,\pi j_{0,1}^2),$ thereby avoiding exceptional values. Noting that $j_{0,1}\sqrt{\pi} \approx 4.262 $, we constrain the wavenumber $\mu \in (0,4.26)$. Similarly to the previous subsection, we first  let $\Omega$ be an ellipse of eccentricity $e$. We plot $F_k(\mu,\Omega)$ as functions of eccentricity in the bottom row of \autoref{fig:fig11b}. Again using \autoref{eq:bound1} we get $F_1(\mu,\Omega) \leq -\mu^2$.

We observe in \autoref{fig:fig11b} that the behaviour of both $G_k$ and $F_k$ is rather unpredictable for wavenumbers outside of the restrictions i.e $\mu_G > 15.1048\cdots$ and $\mu_F > 4.262\cdots$. We also note that local optima are observed for both $G_k,F_k,~k=3,4$ at eccentricty $e \approx 0.855$ (see vertical line in columns 3,4 of \autoref{fig:fig11b}) independent of $\mu$. 


Through extensive numerical experiments, we establish that the disk $\bb D$ is not a minimizer of $F_1(\mu,\Omega)$ (see \autoref{fig:f1_min} below) for $\mu \in (0, \sqrt{\pi}j_{0,1})$. In other words, a Faber-Krahn like inequality with these constraint is not true. As a counterexample, we numerically observe that a simple non-convex family of kite type domains $\tilde\Omega$ whose boundary is parametrized by \autoref{eq:kite} has $F_1(\mu,\tilde\Omega)- F_1(\mu,\bb D) < 0$.

\begin{figure}
		\centering
		\includegraphics[width=0.45\textwidth]{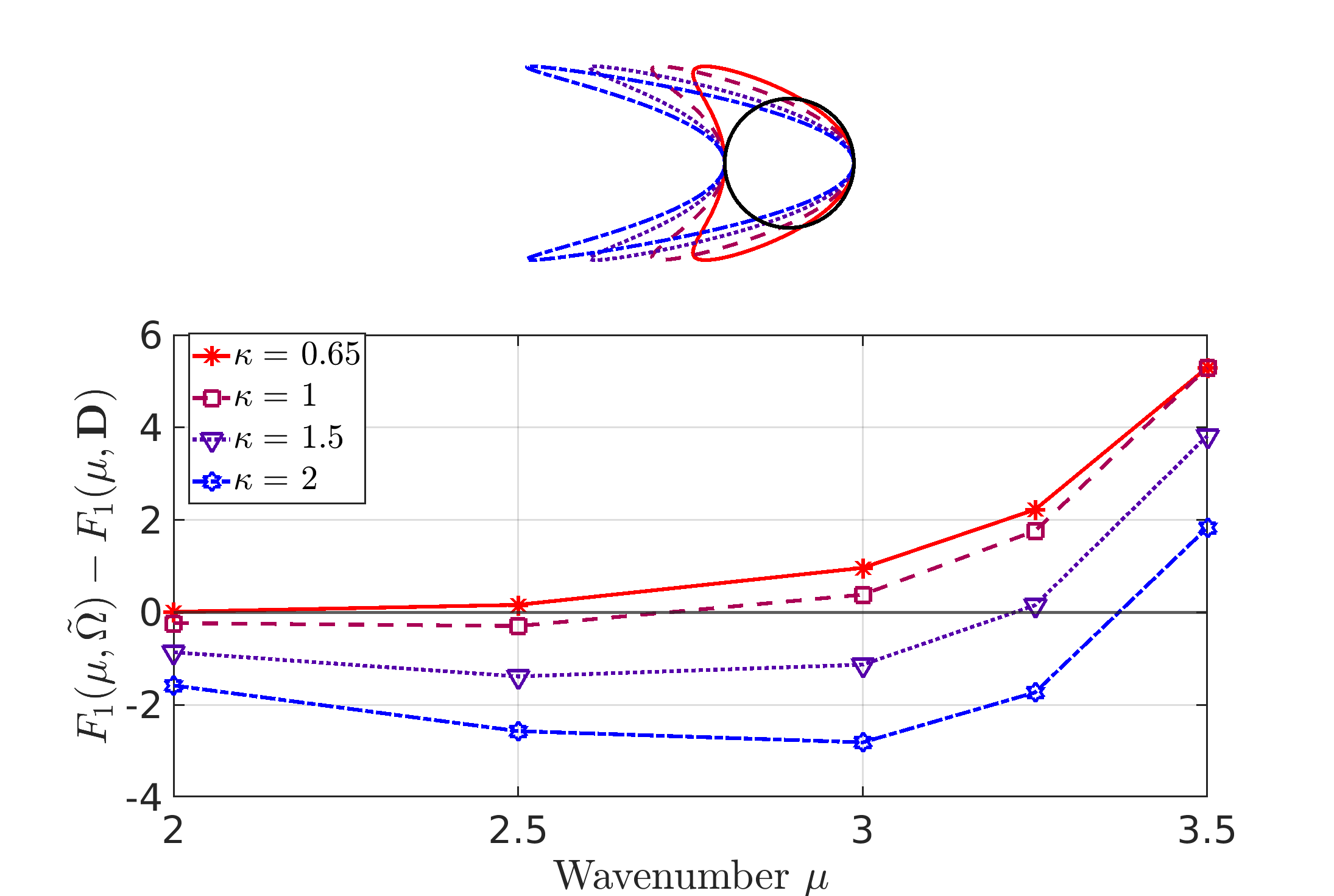}        
	\caption[]                        {{ The unit disk $\mathbb{D}$ does not minimize $F_1(\mu,\Omega) := |\Gamma|\sigma_1\left(\frac{\mu}{\sqrt{|\Omega|}},\Omega\right)$. 
	 Plotted are the differences $F_1(\mu,\tilde{\Omega})-F_1(\mu,\bb D)$ as a function of wavenumber $\mu$, for 4 kite-shaped domains $\tilde{\Omega}$ shown.	The black curve corresponds to the unit disk $\bb D$. See \autoref{eq:kite} for the domain parametrization and shape parameter $\kappa$}.}  
	\label{fig:f1_min}
\end{figure}
The spectral quantity $F_2(\mu, \Omega)$ is optimized by the unit disk $\bb D$ for $\mu \in (0,\lambda_1^*)$. We record numerical experiments using the BIO-MOD approach and prove this result following arguments very similar to \cite{FL20}. 

On the unit disk, when $0\leq\mu \leq \sqrt{\pi }j_{0,1} \approx 4.262$, the eigenvalues in ascending order are 
\[
\sigma_k(\mu,\mathbb{D}) = \mu \frac{J'_{k-1}(\mu)}{J_{k-1}(\mu)}.
\]
Therefore,
\[
F_2(\mu,\bb D) := |\partial\mathbb{D}|\sigma_2\left(\frac{\mu}{\sqrt{|\mathbb{D}|}},\mathbb{D}\right) = 2\mu\sqrt{\pi}\frac{J'_{1}(\mu/\sqrt{\pi})}{J_{1}(\mu/\sqrt{\pi})}. 
\]

We now describe  some numerical experiments concerning the optimization of $F_k(\Omega,\mu)$ over a class of domains whose boundaries are parametrized as  $\Gamma:=\Gamma \equiv r(t) = a_0 + \sum_{j=1}^n a_j \cos jt + b_j \sin jt, t \in[0, 2\pi)$. We set  $a_0 > \sum_j |a_j| + |b_j|$ to
avoid self-intersections of the boundary. Without loss of generality we constrain $|a_j |,|b_j | \leq 0.1$ The numerical optimization is via a gradient-free particle swarm implementation in \texttt{Matlab}, \cite{swarm, MezuraMontes2011ConstrainthandlingIN}. At each iterate $n$ of the optimization algorithm, the objective function $F_k(\mu,\Omega_n)$ for the current domain iterate $\Omega_n$ is computed using the BIO-MOD algorithm. Domain volumes and perimeters are computed using high-accuracy quadrature. 
We performed optimization experiments for $\mu = \pi,4.2$, where $F_2(\pi,\bb D) = 0.541834100559795$ and $F_2(4.2,\bb D) = -5.762316721805390$. We observe that for both values of $\mu$ and $n = 2,4,6,8$ the optimizers converge to (near) disk-like domains $\tilde{\bb D}$ with $F_2(\mu,\tilde{\bb D}) \approx F_2(\mu,\bb D)$.

Motivated by these experiments, we can show \begin{theorem}
	Let $\mu < \sqrt{\pi}j_{0,1}$. If $\Omega$ is a Jordan-Lipschitz domain then $F_2(\mu,\Omega)$ is maximized by the disk $\bb D$.
\end{theorem}
The proof is deferred to the Appendix. We follow a similar optimization approach for $F_3(\mu,\Omega)$, \autoref{newfig14}.
\begin{figure}
	\centering
	\includegraphics[width=0.6\linewidth]{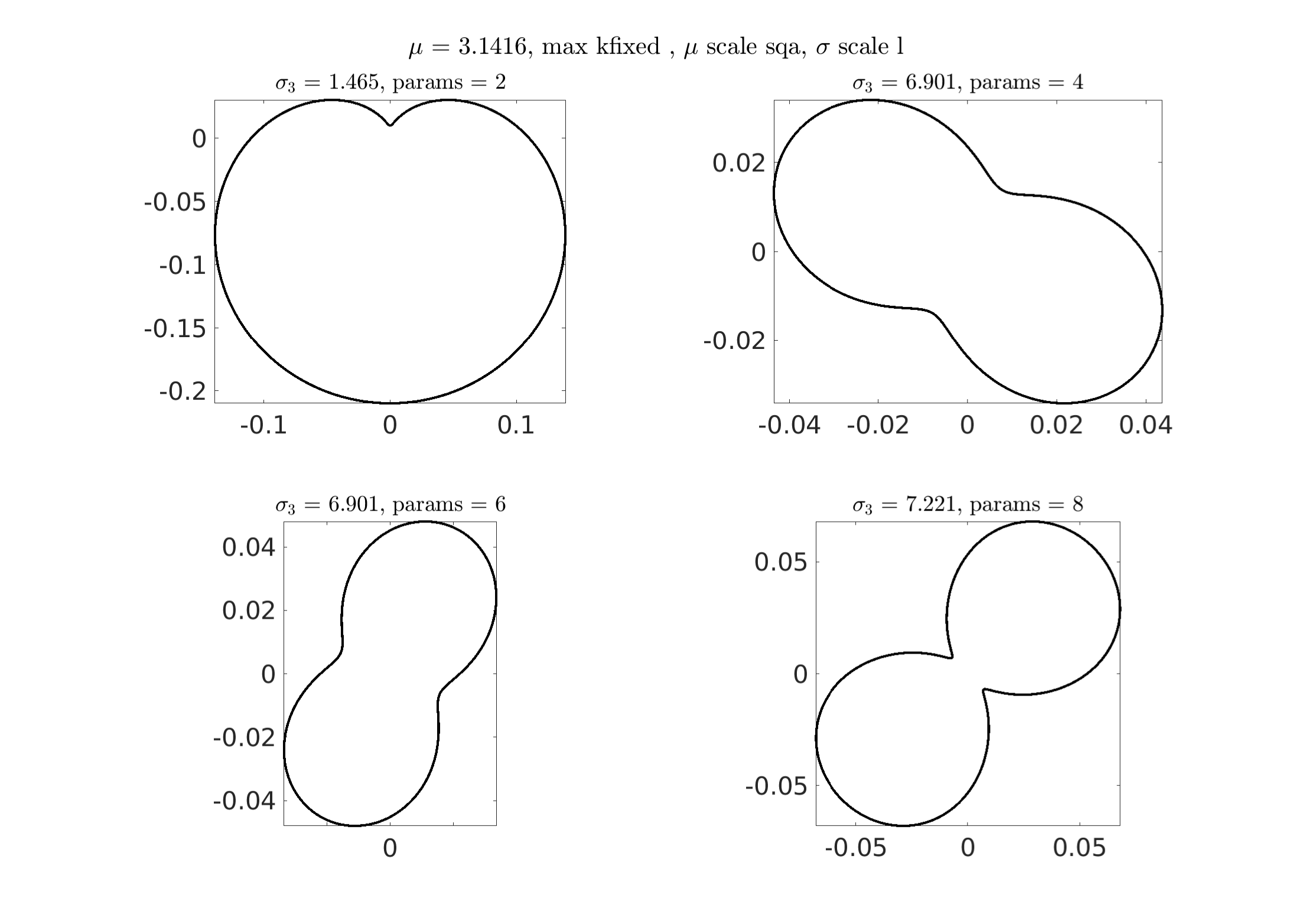}
	\caption{Domains with boundary of form $r(t) = a_0 + \sum_{j=1}^n a_j \cos jt + b_j \sin jt, t \in[0, 2\pi)$   that maximize $F_3(\mu,\Omega)$ for $\mu=3.1416$. As we let the number of shape parameters $n$  increase, the domains approach two disjoint disks. For reference $F_3(\mu,\Omega)=7.273191556454688$ if $\Omega$ is the union of two disjoint disks.}
	\label{newfig14}
\end{figure}

\subsection{Annular domains}
We recall the famous Weinstock inequality for the first Steklov eigenvalue (for the Laplacian, with perimeter constraint) $$ |\Gamma|\sigma(\Omega) \leq |\partial \mathbb{D}| \sigma(\mathbb{D}),$$ fails for multiply-connected domains.

Using the BIO-MOD approach, we can easily compute the Steklov-Helmholtz spectrum on domains of genus 1 (see Section 3.3 in \cite{kshitijMS}). On an annular ring $\bb A_\epsilon$, we hold the outer radius $R_{out}=1$ fixed and vary the inner radius $\epsilon$ from $0.9\to 1e-3$. \autoref{fig:annulusf1} shows a variant of the Weinstock inequality for the first Steklov-Helmholtz eigenvalue (again with perimeter constraint) is also violated. Concretely, we see
$F_1(\mu,A_\epsilon) \geq F_1(\mu,\mathbb{D})$.
\begin{figure}
	\centering
	\includegraphics[width=0.6\linewidth]{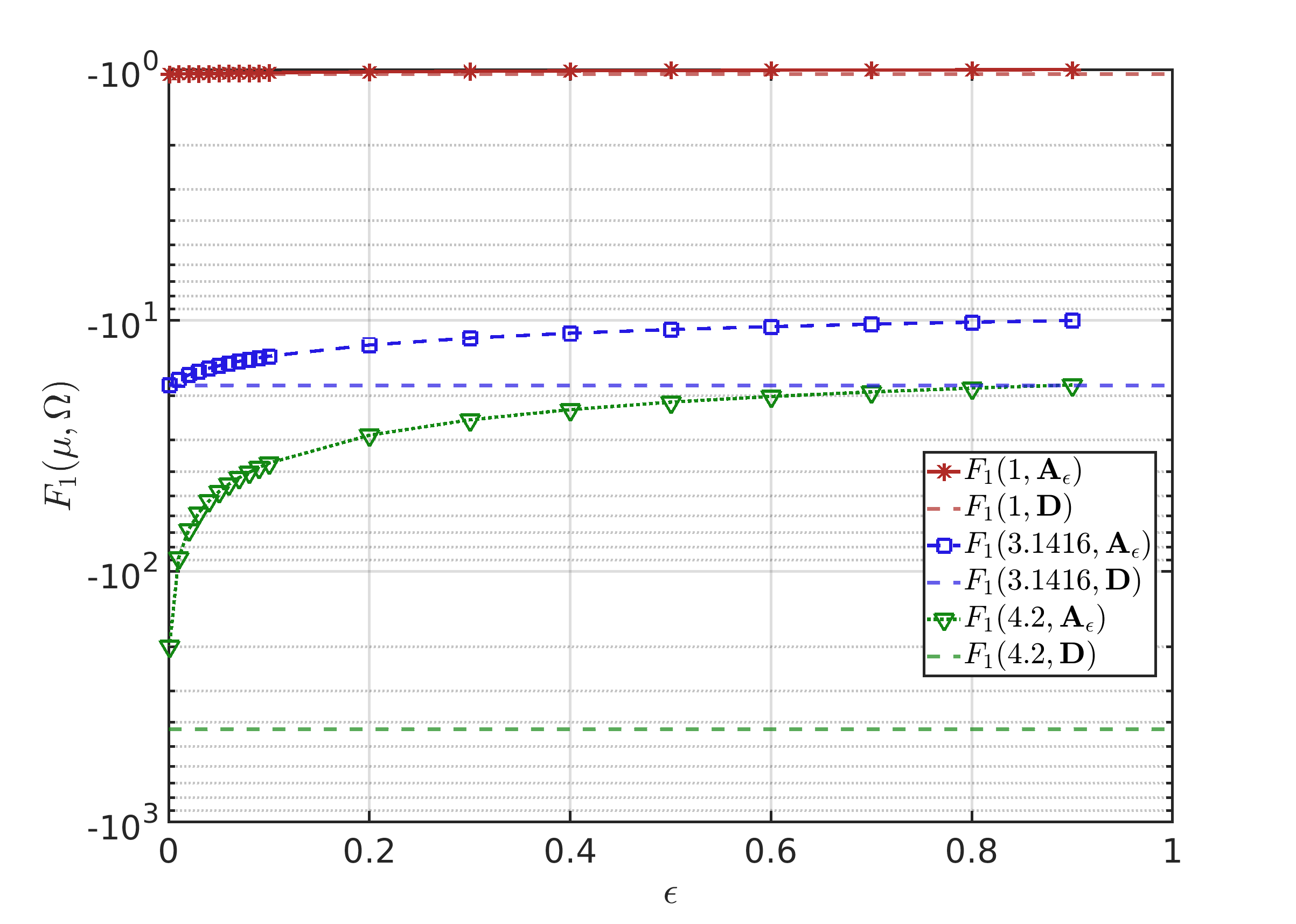}
	\caption{$F_1(\mu,A_\epsilon)$ on annular domains $A_\epsilon$. The inner radius $\epsilon$ is on the x-axis.}
	\label{fig:annulusf1}
\end{figure}

More interesting is the behaviour of $F_2(\mu,A_\epsilon)$, which is shown in \autoref{fig:annulus2}. We propose, therefore:
\paragraph{Open Problem:} For each $\mu \in (0,\pi j_{0,1}^2)$, there is some $\epsilon(\mu)$ such that 
$$ F_2(\mu,A_\epsilon)\geq F_2(\mu,\mathbb{D}), \qquad \forall \epsilon\in (0,\epsilon_\mu).$$ 
\begin{figure}
	\centering
	\includegraphics[width=0.495\linewidth]{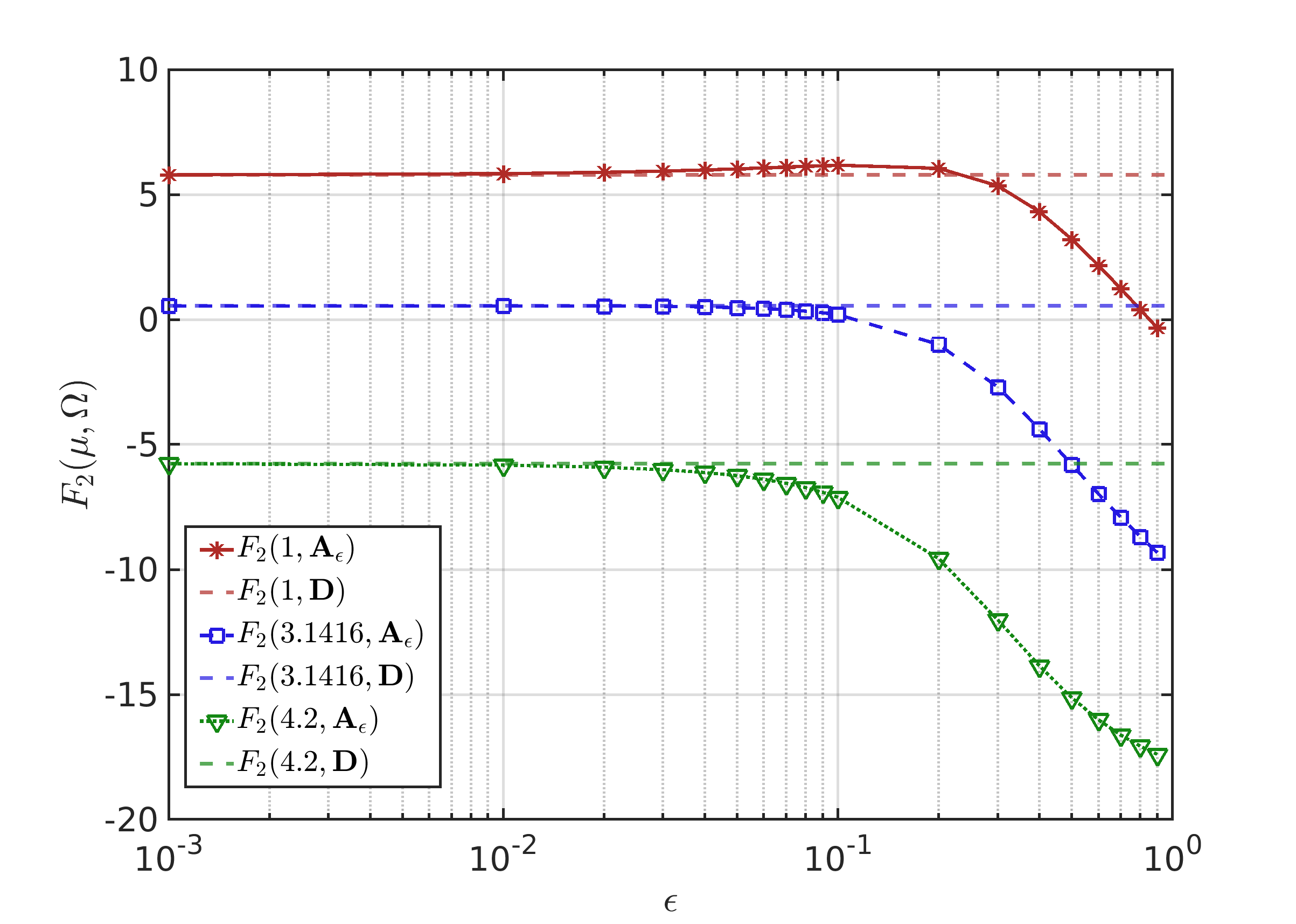}
		\includegraphics[width=0.495\linewidth]{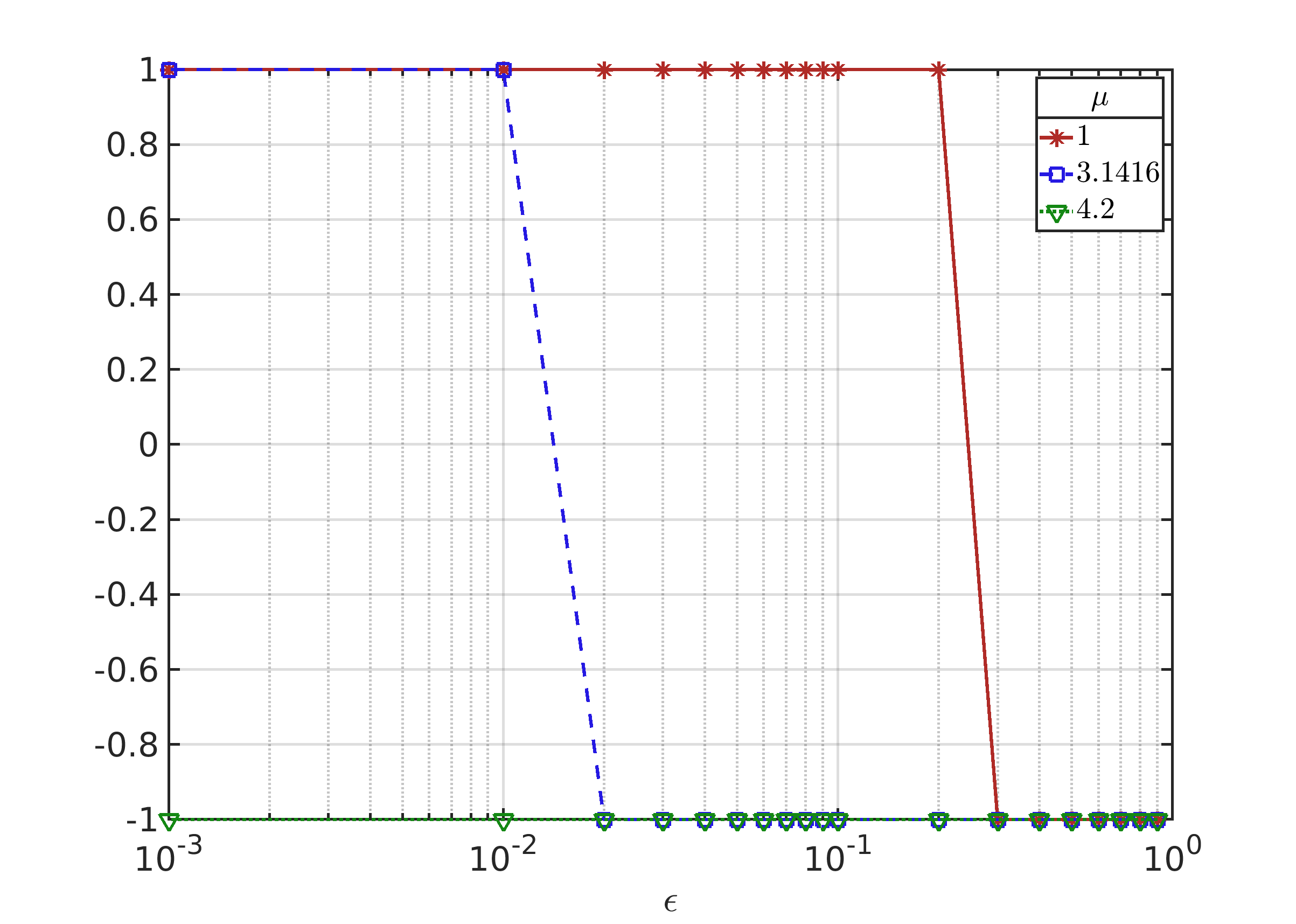}
	\caption{(L) $F_2(\mu, A_\epsilon)$ as a function the inner radius $\epsilon$. Also recorded is $F_2(\mu,\mathbb{D})$. (R) Sign of the function $F_2(\mu,A_\epsilon)-F_2(\mu,\mathbb{D})$.}
	\label{fig:annulus2}
\end{figure}

\section{Conclusion}
In this paper we present a wavenumber-robust approach for computing the Steklov-Helmholtz eigenvalues for planar domains using an indirect layer potential approach. For domains with smooth boundary the eigenvalues are computed with spectral accuracy. We used the approach to study some questions in the spectral geometry of such eigenvalue problems. We present numerical experiments concerning both perimeter-constrained and 
volume-constrained spectral optimization.

Motivated by some of our observations, in future work we will closely study the asymptotic behaviour of the Steklov-Helmholtz eigenfunctions near corners, and incorporate this information to design specialized quadratures for this problem.

\bibliographystyle{siamplain}
\bibliography{NPS.bib}

\section{{Appendix}}

\begin{theorem}
	Let $\mu < \sqrt{\pi}j_{0,1}$. If $\Omega$ is a Jordan-Lipschitz domain then $F_2(\mu,\Omega)$ is maximized by the disk $\bb D$.
\end{theorem}

\noindent
From proposition 7.4.4 of \cite{Levitin23} we know that $\mu^2$ is an eigenvalue of the Robin-Laplacian if and only if $\sigma$ is an eigenvalue of the Steklov-Helmholtz problem. Therefore it is not surprising that under similar scalings, our result is analogous to Theorem~B in \cite{FL20} which concerns the Robin-Laplace eigenvalue problem. To aid our proof we first recall some useful facts about the disk $\bb D$:  
\begin{enumerate}
	\item For the disk, 
	\[
	F_2(\mu,\bb D) = 2\mu\sqrt{\pi}\,\frac{J'_{1}(\mu/\sqrt{\pi})}{J_{1}(\mu/\sqrt{\pi})},
	\]
	and the second eigenvalue has multiplicity $2$.
	\item The corresponding eigenfunctions are
	\[
	u_1(r,\theta) = g(r)\cos\theta,\qquad u_2(r,\theta) = g(r)\sin\theta,
	\]
	where $g(r) = J_1\!\left(\tfrac{\mu}{\sqrt{\pi}}r\right)$.
	\item The radial function $g$ satisfies $g(0)=0 \leq g(1)$ and 
	$g'(r) = \tfrac{\mu}{\sqrt\pi}J'_1\!\left(\tfrac{\mu}{\sqrt{\pi}}r\right)$.  
	A critical point occurs at $r^*$ with $\mu r^*/\sqrt{\pi} = j'_{1,1}$, i.e.\ $r^* = j'_{1,1}\sqrt{\pi}/\mu$.  
	Thus $g'>0$ on $(0,r^*)$ and $g'<0$ on $(r^*,1)$ if $\mu> j'_{1,1}\sqrt{\pi}$, while $g'>0$ on $(0,1)$ if $\mu \leq j'_{1,1}\sqrt{\pi}$. These properties relate to Proposition 5 in \cite{FL20} which is essential for the proof. 
\end{enumerate}

\medskip

\noindent
Next, we recall that in the regime $\mu < \sqrt{\pi}j_{0,1}$ the admissible set in the Rayleigh quotient for the Steklov-Helmholtz eigenvalues is $H^1(\Omega)$. From Theorem~7.4.2 in \cite{Levitin23}, we have
\begin{equation}\label{eqn:rq1}
	\sigma_k\!\left(\tfrac{\mu}{\sqrt{|\Omega|}},\Omega\right) 
	= \min_{\substack{\mathcal{L}\subset H^1(\Omega)\\\dim\mathcal{L}=k}}
	\;\max_{0\neq v\in \mathcal{L}}
	\frac{\int_\Omega |\nabla v|^2-\tfrac{\mu^2}{|\Omega|}|v|^2}{\int_{\Gamma}|v|^2}.
\end{equation}

\noindent
The explicit eigenfunctions $u_1,u_2$ on the disk can be used as test functions after conformal mapping to $\Omega$ and the Rayleigh quotient characterization \eqref{eqn:rq1} allows us to estimate $\sigma_2(\Omega,\mu/\sqrt{|\Omega|})$ using these test functions.  

\begin{proof}
By the center-of-mass lemma (Lemma~6 in \cite{FL20}), a conformal map $f:\bb D\to\Omega$ can be chosen such that $v_i = u_i\circ f^{-1}$ ($i=1,2$) are orthogonal to the first eigenfunction $v_0$ of $\Omega$. Since $v_i$ are smooth and bounded, they lie in $H^1(\Omega)$ and are valid test functions for the Rayleigh quotient~\eqref{eqn:rq1}. Using conformal invariance of the Dirichlet integral, we compute
\begin{align*}
	\sigma_2\!\left(\tfrac{\mu}{\sqrt{|\Omega|}},\Omega\right) 
	&\leq \frac{\int_\Omega |\nabla v_1|^2-\frac{\mu^2}{|\Omega|}|v_1|^2}{\int_{\Gamma}|v_1|^2} \\
	&= \frac{\int_{\bb D} |\nabla u_1|^2-\frac{\mu^2}{|\Omega|}|u_1|^2|f'|^2}{\int_{\Gamma}|v_1|^2}.
\end{align*}
We have, $u_1(r,\theta)=g(r)\cos\theta$ and $v_1=u_1\circ f^{-1}$. Now observe that on the boundary $\Gamma$, 
\[
v_1|_{\Gamma} = u_1(f^{-1}(x)) = g(1)\cos\!\big(\theta(f^{-1}(x))\big),
\qquad x\in\Gamma,
\]
where $\theta$ denotes the polar angle on $\partial\bb D$. Then, 
\[
v_1|_{\Gamma} = g(1)\cos\theta.
\]
Hence the denominator simplifies to $g(1)^2\int_{\Gamma}\cos^2\theta$, and we obtain
\[
\bigg[g(1)^2\!\int_{\Gamma}\cos^2\theta\bigg]\,
\sigma_2\!\left(\tfrac{\mu}{\sqrt{|\Omega|}},\Omega\right) 
\leq \int_{\bb D}\!|\nabla u_1|^2-\tfrac{\mu^2}{|\Omega|}|u_1|^2|f'|^2.
\]
A similar estimate holds for $
v_2|_{\Gamma}=g(1)\sin\theta.
$ Adding the two inequalities (for $v_1$ and $v_2$) yields
\[
g(1)^2F_2(\mu,\Omega) \leq \int_{\bb D}\Big(g'(r)^2 + \frac{1}{r^2}g(r)^2 - \tfrac{\mu^2}{|\Omega|}g(r)^2|f'|^2\Big).
\]
For the disk, equality holds:
\[
g(1)^2F_2(\mu,\bb D) = \int_{\bb D}\Big(g'(r)^2 + \frac{1}{r^2}g(r)^2 - \tfrac{\mu^2}{|\Omega|}g(r)^2\Big).
\]
Thus the theorem reduces to proving
\[
\int_{\bb D} |f'|^2 g(r)^2 \;\geq\; \int_{\bb D} g(r)^2,
\]
which is precisely the inequality established in Theorem~B of \cite{FL20}. Point~3 above ensures applicability of the proof of this inequality.
\end{proof}

\end{document}